\documentclass[10pt, reqno]{amsart}
\usepackage{epsfig, amssymb, amsmath, version}
\usepackage{amssymb, version, graphicx, fancybox, mathrsfs, multirow}
\usepackage{epstopdf}
\usepackage{url, hyperref}
\usepackage{bm}
\usepackage{subcaption}

\usepackage{float}
\usepackage{enumitem}
\usepackage{multirow}

\usepackage{color, xcolor}
\usepackage{cases}
\usepackage{mathtools}
\usepackage{extarrows}
\usepackage{bbm}

\usepackage{booktabs}

\catcode`\@=11 \theoremstyle{plain}
\@addtoreset{equation}{section}   

\@addtoreset{figure}{section}
\renewcommand\thefigure{\thesection. \@arabic\c@figure}
\renewcommand{\thefigure}{\arabic{section}.\arabic{figure}}
\newtheorem{thm}{\bf Theorem}

\newenvironment{theorem}{\begin{thm}} {\end{thm}}
\newtheorem{cor}{\bf Corollary}

\newtheorem{lmm}{\bf Lemma}

\newenvironment{lemma}{\begin{lmm}}{\end{lmm}}

\theoremstyle{remark}
\newtheorem{rem}{\bf Remark}[section]

\theoremstyle{definition}

\numberwithin{table}{section}
\usepackage{appendix}

\allowdisplaybreaks[4]

\renewcommand \wedge \times

\begin{document}
	\bibliographystyle{plain}
	\graphicspath{{./figures/}}
\title[A numerical study for tempered time-fractional]
	 {A numerical study for tempered time-fractional advection-dispersion equation on graded meshes}
	\author[L. Huang,\quad L. Li,\quad S. Lü]
{\;\; Liangcai Huang${ }^{\dagger}$,\quad Lin Li${ }^{\ddagger}{ }^{\S}$,\quad  Shujuan Lü${ }^{\dagger}$ \;\;}
	
	\thanks{${ }^{\dagger}$School of Mathematical Sciences, Beihang University, Beijing 102206, China. (hlcsy@buaa.edu.cn, lsj@buaa.edu.cn).  \\
    \indent ${ }^{\ddagger}$School of Mathematics and Physics, University of South China, Hengyang 421001, China. (lilinmath@usc.edu.cn).\\
     \indent ${ }^{\S}$Corresponding author.
			}
\begin{abstract}
In this paper, we develop a second-order accurate time-stepping scheme for the tempered time-fractional advection–dispersion equation based on a sum-of-exponentials (SOE) approximation to the convolution kernel involved in the fractional derivative. To effectively resolve the weak initial-time singularity at $t=0$, graded temporal meshes are employed. A fully discrete scheme is constructed by coupling the proposed half-time-level temporal discretization with a finite difference method in space. Compared with the classical L1 scheme, the proposed SOE-based method achieves the same global convergence order while reducing both storage requirements and computational cost. Specifically, the storage demand is reduced from $\mathcal{O}(MN)$ to $\mathcal{O}(MN_{\mathrm{exp}})$, and the computational complexity is lowered from $\mathcal{O}(MN^2)$ to $\mathcal{O}(MN N_{\mathrm{exp}})$, where $M$ and $N$ denote the numbers of spatial and temporal grid points, respectively, and $N_{exp}$ is the number of exponential terms used in the SOE approximation. The unique solvability, stability
and accuracy of the resulting scheme are rigorously analyzed. Several numerical results are presented to confirm the sharpness of the error analysis and to demonstrate the efficiency of the proposed method.

\end{abstract}
\maketitle
\section{Introduction}
Tempered fractional calculus is a refined extension of classical fractional calculus, designed to model complex phenomena characterized by anomalous diffusion and nonlocal transport. By introducing an exponential tempering factor into power-law kernels, it effectively overcomes the infinite-moment difficulty inherent in classical fractional models while preserving their nonlocal memory features. The concept of truncated Lévy flights was first proposed by Koponen \cite{koponen1995} as a smooth modification of the model introduced by Mantegna \cite{mantegna1994}. Subsequently, Cartea and del Castillo-Negrete \cite{cartea2007} developed the tempered fractional calculus framework to ensure the finiteness of moments in Lévy distributions arising from power-law waiting times. Since then, tempered fractional models have found widespread applications in many fields, including finance \cite{carr2002, carr2003, zhou2024}, geophysics \cite{meerschaert2008, xia2013, zhang2011, zhang2012}, and groundwater hydrology \cite{zhang2014}. By capturing both memory effects and non-Gaussian transport behavior, tempered fractional calculus significantly enhances the modeling and prediction of complex systems.

The application of tempered fractional models to real-world problems requires efficient and accurate numerical methods for solving tempered fractional differential equations. In recent years, tempered fractional operators have been incorporated into various mathematical models, such as the substantial fractional equation \cite{huang2018}, the Bloch equation \cite{feng2022}, the tempered fractional Burgers equation \cite{dwivedi2025, wang2022}, the tempered fractional Jacobi equation \cite{zhao2023}, and the tempered fractional Feynman–Kac equation \cite{zhao2024}. Among these, the numerical solution of the Caputo time-tempered fractional diffusion equation (TTFDE) has attracted considerable attention \cite{dwivedi2024, fenwick2024, krzyzanowski2024, morgado2019, safari2022, zhao2020, zhou2024}, which is defined as
\begin{align}
{ }_0^C D_t^{\alpha, \lambda} u = \frac{\partial^2 u}{\partial x^2} - \frac{\partial u}{\partial x} + f(x,t), \quad (x,t) \in \Omega \times (0,T], \label{no1}
\end{align}
where $\lambda > 0$ and $0 < \alpha < 1$. The Caputo tempered fractional derivative is defined by
\begin{align}
{ }_0^C D_t^{\alpha, \lambda} u(t) = \frac{e^{-\lambda t}}{\Gamma(1-\alpha)} \int_0^t \frac{(e^{\lambda s} u(s))'}{(t-s)^\alpha} ds, \label{tempered}
\end{align}
where $\Gamma(\cdot)$ denotes the Gamma function. When $\lambda = 0$, this operator reduces to the classical Caputo fractional derivative
$$
{ }_0^C D_t^\alpha u(t)=\frac{1}{\Gamma(1-\alpha)} \int_0^t \frac{u'(s)}{(t-s)^{\alpha}}  d s. 
$$
A variety of numerical methods have been proposed for solving tempered fractional diffusion-type equations(TFDE). Chen \cite{chen2018} define a family of generalized Laguerre functions and derive some useful formulas of tempered fractional integrals/derivatives. Morgado \cite{morgado2019} developed a finite difference method for tempered fractional advection–diffusion equations and achieved a temporal accuracy of $\min\{2-\alpha, r\alpha\}$ and second-order spatial accuracy using a graded temporal mesh. Zhao \cite{zhao2020} proposed a finite difference method for time-fractional subdiffusion equations, obtaining a temporal accuracy of $(2-\alpha)$ and second-order spatial accuracy for smooth solutions on uniform grids. Safari \cite{safari2022} employed a local discontinuous Galerkin method for time–space tempered fractional diffusion equations and obtained a temporal accuracy of $\min\{2-\alpha, r\delta\}$ with high-order spatial accuracy. Zhou \cite{zhou2024} developed a fast finite difference method for tempered time-fractional diffusion–reaction equations, achieving a temporal accuracy of $\min\{2-\alpha, r\alpha\}$ and fourth-order spatial accuracy. More recently, Krzyżanowski \cite{krzyzanowski2024} and Fenwick \cite{fenwick2024} proposed finite difference and fast finite difference methods, respectively, both achieving a temporal accuracy of $(2-\alpha)$ and second-order spatial accuracy for smooth solutions. Dwivedi \cite{dwivedi2024} developed a fast scheme for a two-dimensional tempered fractional reaction–advection–subdiffusion equation, obtaining a temporal accuracy of $(1+\alpha)$ and fourth-order spatial accuracy. 

In this paper, we consider the tempered time-fractional advection–dispersion equation (TTFADE) proposed by Xia \cite{xia2013} and Meerschaert \cite{meerschaert2008}, compared with the TTFDE, the TTFADE additionally involves a first-order time derivative term, which plays an important role in modeling and restoring groundwater environments. The TTFADE model is given by:
\begin{subequations}\label{Equation}
\begin{align}
&\frac{\partial u}{\partial t} + { }_0^C D_t^{\alpha, \lambda} u
= \frac{\partial^2 u}{\partial x^2} - \frac{\partial u}{\partial x} + f(x,t),
& &(x,t) \in \Omega \times (0,T], \label{Equation1} \\
&u(0,t) = 0,\quad u(L,t) = 0,
& &t \in (0,T], \label{Equation2} \\
&u(x,0) = \phi(x),
& &x \in \Omega, \label{Equation3}
\end{align}
\end{subequations}
where $\lambda > 0$ and $0 < \alpha < 1$. The truncation error analysis and convergence analysis of the proposed scheme are carried out under the following regularity assumptions \cite{chen2019,feng2022,wang2023,safari2022,stynes2017,zhang2021}. We assume that the problem \ref{Equation} has a unique solution $u$ and there is a constant $C > 0$ such that
\begin{subequations}\label{regularity}
\begin{align}
&|\partial_x^k u(x, t)| \leq C,
&& (x, t) \in \Omega \times [0, T],\;k=0,1,2,3,4, \label{regularity1} \\
&|\partial_t^l u(x, t)| \leq C(1 + t^{\delta-l}),
&& (x, t) \in \Omega \times (0, T], \; l=0,1,2,3, \label{regularity2} \\
&|\partial_t^2 \partial_x^p u(x, t)| \leq C(1 + t^{\delta-2}),
&& (x, t) \in \Omega \times (0, T], \; p=1,2,\;1<\delta<2. \label{regularity3}
\end{align}
\end{subequations}
For this model, before introducing our approach, it is necessary to clarify the underlying motivations and position our contribution relative to existing work. Cao \cite{cao2020} proposed a fast finite difference/finite element method based on graded temporal meshes, achieving first-order accuracy in time and high-order accuracy in space for problems with weakly singular solutions. In contrast, the method developed in our work attains second-order temporal accuracy under the same low-regularity setting. Qiao \cite{qiao2022} constructed a second-order backward differentiation formula coupled with a second-order convolution quadrature, yielding temporal accuracy of order $(1+\alpha)$ and second-order spatial accuracy. However, their analysis relies on solution smoothness assumptions, whereas our results are established for weakly singular solutions. In addition, several related studies focus primarily on smooth solutions and on equations with $\lambda = 0$. For instance, Liu \cite{Liu2018} investigated variable-order fractional equations using finite difference methods, obtaining first-order temporal and second-order spatial accuracy under smoothness assumptions. Ravi Kanth \cite{RaviKanth2018} proposed a finite difference scheme that achieves $(2-\alpha)$-th order convergence in time and second-order convergence in space for smooth solutions. Nong \cite{nong2024} developed a fast $L2-1_{\sigma}$ scheme and proved second-order convergence in both time and space, again for smooth solutions.

It is well known that the principal numerical challenges in tempered fractional models arise from the presence of a weakly singular kernel and the exponential tempering factor in the fractional operator. Moreover, it is highly desirable to design efficient numerical schemes that maintain high accuracy even when the solution exhibits low regularity near the initial time \eqref{regularity}. In this setting, most existing methods achieve at most $(2-\alpha)$-th order accuracy in time. To address these challenges, we develop in this paper a finite difference scheme for the TTFADE \eqref{Equation} based on a sum-of-exponentials (SOE) approximation. The main features and contributions of our work can be summarized as follows.
\begin{enumerate}[leftmargin=15pt]
    \item The proposed scheme employs graded temporal meshes to better capture the initial singular behavior and evaluates the discrete operator at half-time levels to enhance accuracy.
    \item  We provide rigorous proofs of stability and convergence for the scheme, establishing second-order accuracy in both time and space under the weak singularity assumptions.
    \item Three numerical cases are presented to verify both the accuracy and computational efficiency of the method. Compared with several existing approaches, the proposed scheme produces numerical results more efficiently and with reduced computational cost (see the examples in Sect.~\ref{section4}).
\end{enumerate}

The rest of the paper is organized as follows. In the next Sect.\ref{charpter2}, we present the discrete scheme and error estimation to the tempered Caputo fractional derivative. In Sect.\ref{section3}, we construct a fully discrete
scheme for the tempered time-fractional advection–dispersion equation with weakly singular solution and present the well-posedness of the scheme. In Sect.\ref{section4}, some numerical results are presented to confirm
the theoretical analysis results and demonstrate the computational efficiency of the method. Finally, concluding remarks are given in Sect.\ref{section5}.
\section{The discrete scheme and error estimation to the Tempered Caputo fractional Derivative }
\label{charpter2}
In this section, we develop a discrete scheme based on the sum-of-exponentials (SOE) approximation of the kernel function for the tempered Caputo fractional derivative defined in \eqref{tempered}, and derive rigorous error estimates for the proposed method.

We first introduce the graded temporal grid
\begin{align}
T_N = \{ t_n \mid t_n = T (n/N)^r, \; 0\le n \le N\} \label{graded}
\end{align}
on the interval $(0, T]$, where $N$ is a positive integer, $r \ge 1$, and the variable time step is defined by $\tau_n = t_n - t_{n-1}$. It is well known that graded meshes are an effective tool for resolving weak initial singularities in fractional evolution equations (see, e.g., \cite{Brunner1985,cao2020,Liao2018,safari2022}).

To achieve second-order accuracy in time for the fully discrete scheme, we introduce the midpoint time level
\begin{align}
    \bar{t}_n = \frac{1}{2}(t_n + t_{n+1}),
\end{align}
and approximate the temporal derivative at $\bar{t}_n$. The basic idea is to decompose the fractional derivative at $\bar{t}_n$ into a history part and a local part, namely,
\begin{align}
{ }_0^C D_t^{\alpha,  \lambda} u(\bar{t}_n)\coloneqq C_h(\bar{t}_n)+C_l(\bar{t}_n) , \label{fast_define}
\end{align}
where
\begin{align}
    C_h(\bar{t}_n)&=\frac{e^{-\lambda \bar{t}_n}}{\Gamma(1-\alpha)}\int_0^{t_{n}} (\bar{t}_n-s)^{-\alpha} d(e^{\lambda s} u(s)),\label{history}\\
    C_l(\bar{t}_n)&=\frac{e^{-\lambda \bar{t}_n}}{\Gamma(1-\alpha)}\int_{t_{n}}^{\bar{t}_n} (\bar{t}_n-s)^{-\alpha} d(e^{\lambda s} u(s)).\label{local}
\end{align}

For \eqref{history}, it is known that the weakly singular kernel $t^{-1-\alpha}$, $\alpha \in (0,1)$, can be efficiently approximated by a sum-of-exponentials (SOE) on any interval $[\Delta t, T]$ with $\Delta t > 0$ (see, e.g., \cite{beylkin2005,cao2020,zheng2025}). More precisely, for any prescribed accuracy $\varepsilon > 0$, there exist positive real numbers ${\omega_l, s_l}$, $l=1,\dots,N_{exp}$ such that
\begin{align}
    |t^{-1-\alpha}-\sum_{l=1}^{N_{exp}}\omega_le^{-s_lt}|\leq \varepsilon, \quad \forall t\in [\Delta t,T].\label{SOE}
\end{align}
 Using integration by parts to \eqref{history} and approximating the kernel $(\bar t_n-s)^{-1-\alpha}$ by \eqref{SOE} give
\begin{align}\label{Ch}
C_h(\bar{t}_n) 
& \approx \frac{1}{\Gamma(1-\alpha)}[e^{-\lambda {\tau_{n+1}}/{2}} u(t_{n}){({\tau_{n+1}}/{2})^{-\alpha}}-{e^{-\lambda \bar{t}_n}}{\bar{t}_n^{-\alpha}} u(t_0)-\alpha \sum_{i=1}^{N_{\text {exp }}} \omega_i U_{h i s,  i}(t_n)],
\end{align}
where 
\begin{align}
    U_{h i s,  i}(t_n)=\int_0^{t_{n}} e^{-(\lambda+s_i)(\bar{t}_n-s)} u(s) d s.
\end{align}  
The key for the success of the approach lies in the fact that there exists the recurrence
\begin{align}
U_{h i s,  i}(t_n)=e^{-(\lambda+s_i) \frac{\tau_n+\tau_{n+1}}{2}} U_{h i s,  i}(t_{n-1})+\int_{t_{n-1}}^{t_{n}} e^{-(\lambda+s_i)(\bar{t}_n-s)} u(s) d s . \label{Uequa}
\end{align}
To evaluate the integral term, we approximate $u$ on $[t_{n-1}, t_n]$ by piecewise linear interpolation
\begin{align}
    L_{1,  n} u = \frac{s-t_{n-1}}{\tau_{n}}u(t_{n})+\frac{t_{n}-s}{\tau_{n}}u(t_{n-1}),\label{interpolation}
\end{align}
which leads to
\begin{align}
    U_{his, i}(t_n) \approx \sum_{j=0}^{n-1} e^{-(\lambda+s_i)(\bar{t}_n-\bar{t}_{n-j})}(\lambda_{i,  n-j}^1 u(t_{n-j})+\lambda_{i,  n-j}^2 u(t_{n-1-j})),
\end{align}
where 
\begin{equation}\label{lambda}
    \begin{aligned}
        \lambda_{i, n}^1
        &=\frac{e^{-(\lambda+s_i){\tau_{n+1}}/{2}}}{(\lambda+s_i)^2\tau_{n}}(e^{-(\lambda+s_i) \tau_{n}}-1+(\lambda+s_i) \tau_{n}), \\
        \lambda_{i, n}^{2}
        &=\frac{e^{-(\lambda+s_i){\tau_{n+1}}/{2}}}{(\lambda+s_i)^2\tau_{n}}(1-e^{-(\lambda+s_i) \tau_{n}}-e^{-(\lambda+s_i) \tau_{n}}(\lambda+s_i) \tau_{n}),
    \end{aligned}
\end{equation}
Substituting the above approximation into \eqref{Ch} and rearranging the terms, we obtain
\begin{equation}
\begin{aligned}
    C_h(\bar{t}_n) \approx &\frac{1}{\Gamma(1-\alpha)}({e^{-\lambda {\tau_{n+1}}/{2}} u(t_{n})}{({\tau_{n+1}}/{2})^{-\alpha}}-{e^{-\lambda \bar{t}_n}}{\bar{t}_n^{-\alpha}} u(t_0) \\
     &-\sum_{l=1}^{n-1}(a_{n-l, n}+b_{n-1-l, n})u(t_l)-a_{0, n}u(t_{n})-b_{n-1, n}u(t_0)). 
\end{aligned}
\end{equation}
where
\begin{equation}\label{ab}
\begin{aligned}
&a_{j,  n}=\alpha \sum_{i=1}^{N_{\exp }} \omega_i e^{-(\lambda+s_i)(\bar{t}_n-\bar{t}_{n-j})} \lambda_{i,  n-j}^1,  \\
&b_{j,  n}=\alpha \sum_{i=1}^{N_{\exp }} \omega_i e^{-(\lambda+s_i)(\bar{t}_n-\bar{t}_{n-j})} \lambda_{i,  n-j}^2. 
\end{aligned}
\end{equation}

For \eqref{local}, we employ the linear interpolation $L_{1,k+1}u$ of the function $u$ defined in \eqref{interpolation} on $[t_k, t_{k+1}]$ (see, e.g., \cite{cao2020,Liu2018,sun2006,stynes2017,zhou2024}). Then we write
\begin{align} \label{Cl}
C_l(\bar{t}_n) & \approx \frac{e^{-\lambda \bar{t}_n}}{\Gamma(1-\alpha)} \int_{t_{n}}^{\bar{t}_n} (\bar{t}_n-s)^{-\alpha}(e^{\lambda s} L_{1,  n+1} u(s))^{\prime} d s .
\end{align}
In particular, for the half interval $[t_n, \bar{t}_n]$, we still use the linear interpolation over the entire interval $[t_n, t_{n+1}]$, which is different from the treatment adopted in \cite{Liu2018}. Approximating the derivative term by
\begin{align}
    (e^{\lambda s} L_{1,  n+1} u(s))^{\prime}\approx \frac{e^{\lambda \bar{t}_n} L_{1,  n+1} u(\bar{t}_n)-e^{\lambda t_{n}} L_{1,  n+1} u(t_{n})}{{\tau_{n+1}}/{2}}.\label{local2}
\end{align}
Substituting (\ref{local2}) into (\ref{Cl}), we finally obtain
\begin{align}
    C_l(\bar{t}_n) \approx \frac{1}{\Gamma(1-\alpha)}(\frac{u(t_{n})+u(t_{n+1})}{2^{1-\alpha}(1-\alpha)\tau_{n+1}^\alpha} -\frac{2e^{-\lambda {\tau_{n+1}}/{2}}}{2^{1-\alpha}(1-\alpha)\tau_{n+1}^\alpha} u(t_{n})).
\end{align}

Based on the above approximation strategy and detailed derivation, we derive the following SOE discrete operator for the tempered Caputo derivative.
\begin{equation} \label{fastoperator}
\begin{aligned}
{ }_0^C D_t^{\alpha,  \lambda} u(\bar{t}_n)\approx{ }_0^{FC}D_t^{\alpha,  \lambda} u(\bar{t}_n) =  &\frac{\tau_{n+1}^{-\alpha}}{2^{1-\alpha}\Gamma(2-\alpha)}[u(t_{n+1})+(1-2\alpha e^{-\lambda \tau_{n+1}/2})u(t_{n})] \\
    & -\frac{1}{\Gamma(1-\alpha)}[a_{0, n}u(t_n)+\sum_{l=1}^{n-1}(a_{n-l, n}+b_{n-1-l, n})u(t_l)\\
    &+(b_{n-1, n}+e^{-\lambda \bar{t}_n}\bar{t}_n^{-\alpha})u(t_0)].
\end{aligned}
\end{equation}
Next, we analyze the error of the SOE discrete operator ${}_0^{FC}D_t^{\alpha,\lambda} u(\bar{t}_n)$ defined in \eqref{fastoperator}. Before presenting the final error estimate, we first introduce the L1 approximation ${}_0^C\mathbb{D}_t^{\alpha, \lambda} u(\bar{t}_n)$ \cite{cao2020,Liao2018,stynes2017} to the tempered Caputo derivative and establish the following auxiliary lemma.
\begin{align}
    _0^C\mathbb{D}_t^{\alpha, \lambda} u(\bar{t}_n)=\frac{e^{-\lambda \bar{t}_n}}{\Gamma(1-\alpha)}\int_{t_k}^{\bar{t}_n}(\bar{t}_n-s)^{-\alpha}(e^{\lambda s}L_{1, k+1}u(s))^{\prime}ds,\label{L1operator}
\end{align}
where $L_{1, k+1}u(s)$ is the linear interpolation function of $u(s)$ on $[t_k, t_{k+1}]$.
\begin{lemma}\label{Ru}
Under the regularity assumptions \eqref{regularity}, there exists a constant $K$, independent of $n$ and $N$, such that
    \begin{align}
        | {}^{LC}Ru(\bar{t}_n)|&= |{}^C_0 D^{\alpha, \lambda}_t u(\bar{t}_n)- {}^C_0\mathbb{D}_t^{\alpha,  \lambda} u(\bar{t}_n)|\leq 
        \begin{cases} 
        KN^{-r(\delta+1-\alpha)},&n=0,\\
        K(n+1)^{-\min\{2-\alpha, r(1+\delta)\}}, &n\geq 1.
        \end{cases}\label{L1half}
\end{align}
\end{lemma}
\begin{proof}We decompose the truncation error into a history part and a local part. Using integration by parts yields
    \begin{align}\label{R1}
        {}^{LC}Ru(\bar{t}_n)
        = &K(\sum_{k=0}^{n-1}R_{k}(\bar{t}_n)+R_{n}(\bar{t}_n)),
    \end{align}
where $K$ is a constant independent of $n$ and $N$ and 
\begin{align}
    R_{k}(\bar{t}_n)=&\int_{t_k}^{t_{k+1}}{(\bar{t}_n-s)^{-1-\alpha}}e^{\lambda s-\lambda \bar{t}_n}(u(s)-L_{1, k+1}u(s))ds,\quad 0\leq k\leq n-1, \label{hist1}\\
        R_{n}(\bar{t}_n)=&\int_{t_n}^{\bar{t}_n}\alpha (\bar{t}_n-s)^{-\alpha}(e^{\lambda s-\lambda \bar{t}_n}u(s)-e^{\lambda s-\lambda \bar{t}_n}L_{1, n+1}u(s))^{\prime} ds.\label{local1}
\end{align}

First, for the special case \(n=0\), applying Taylor expansion on the interval \(s \in [t_0, \bar{t}_0]\), there exist \(\xi_1, \xi_2 \in [t_0, \bar{t}_0]\) such that
\begin{align}
    &|u(s)-L_{1, 1}u(s)|
    =|{(t_1-s)(s-t_0)}u'(\xi_1)+{(s-t_0)(s-t_1)}u'(\xi_2)|/{\tau_1}
    ,\label{1e}\\
    &u'(s)-(u(t_1)-u(t_0))/{\tau_1}=(\int_s^{t_{1}}u''(t)(t_{1}-t)dt-\int_s^{t_{0}}u''(t)(t_{0}-t)dt)/{\tau_1}.\label{2e}
\end{align}
Using the identities \eqref{1e}–\eqref{2e}, the regularity assumptions \eqref{regularity} and the identity
\begin{align*}
    \int_0^x (x-s)^{-\alpha}s^{\delta-2} ds=x^{\delta-\alpha-1} B(1-\alpha,\delta-1),
\end{align*}
where $B(\cdot,\cdot)$ is Beta function, we obtain
\begin{align}
|{}^{LC}Ru(\bar{t}_0)|=|R_{0}(\bar{t_0})|
\leq &KN^{-r(\delta+1-\alpha)}.\label{gamma}
\end{align}

Next, for $n \ge 1$, we estimate the error terms in \eqref{R1} by considering four separate cases for $k$, namely,
\begin{align*}
    k=0, \quad 1 \le k \le \lceil  \frac{n+1}{2} \rceil - 1,\quad \lceil \frac{n+1}{2} \rceil \le k \le n - 1,\quad k = n.
\end{align*}
\begin{description}[leftmargin=0pt, labelindent=0pt]
\item [Case 1 ($k=0$)] Employing the inequality established in~\cite{cao2020},
\begin{align}
(({t_{n+1}-t_1})/{t_1})^{-1} \le C (n+1)^{-r}, \label{time}
\end{align}
and combining it with~\eqref{hist1}, we obtain
    \begin{align}
|R_{0}(\bar{t}_n)| \leq C(\bar{t}_n-t_1)^{-1-\alpha}(\int_{t_0}^{t_1} s^\delta+s t_1^{\delta-1} d s) \leq K(n+1)^{-r(\delta+1)}.
\end{align}
    \item[Case 2 ($ 1\leq k\leq \lceil (n+1)/{2}\rceil-1$) ] By the interpolation error formula  which is different from (\ref{1e}) because of the singularity at $t_0$, there exists $\xi_{k+1} \in (t_k, t_{k+1})$ such that
\begin{align}
|u(s) - L_{1,k+1}u(s)|
= \tfrac12 |u''(\xi_{k+1})| (s-t_k)(t_{k+1}-s)
\le K_1 \tau_{k+1}^2 t_k^{\delta-2}. \label{taylor1}
\end{align}
Moreover, using the graded mesh property given in \cite{cao2020},
\begin{align}
    \tau_{k+1}\leq C TN^{-r}k^{r-1},\label{step}
\end{align}
we obtain
    \begin{align}
    |R_{k}(\bar{t}_n)| \leq K_2\tau_{k+1}^3t_k^{\delta-2}(\bar{t}_n-t_{k+1})^{-1-\alpha}\leq K_3k^{r(1+\delta)-3}(n+1)^{-r(1+\delta)}.\label{sumup}
    \end{align}
Summing up \eqref{sumup} and applying the inequality from \cite{gracia2018} yields
\begin{align}
    \sum_{k=1}^{\lceil (n+1)/{2}\rceil-1}|R_{k}(\bar{t}_n)| \leq C\begin{cases}(n+1)^{-r(\delta+1)},&\text{ if }r(1+\delta)<2, \\(n+1)^{-2}\ln n,&\text{ if }r(1+\delta)=2, \\(n+1)^{-2},&\text{ if }r(1+\delta)>2. \end{cases}
\end{align}
\item[Case 3 ($\lceil (n+1)/{2}\rceil\leq k \leq n-1$) ] Noting that $t_k \ge 2^{-r} t_{n+1}$, combining it and \eqref{taylor1} together, and plugging the result into \eqref{hist1}, we immediately obtain 
\begin{align}
\sum_{k=\lceil (n+1)/{2}\rceil}^{n-1}|R_{k}(\bar{t}_n)|\leq K_4\tau_{n}^2t_{n+1}^{\delta-2}(\bar{t}_n-t_{n})^{-\alpha}\leq K_5(n+1)^{-(2-\alpha)}.
\end{align}
\item[Case 4 ($k=n$)]By Taylor expansion, for some $\xi \in [t_n, t_{n+1}]$,
\begin{align}
|u'(s) - L'_{1,n+1}u(s)|
= | \int_s^\xi u''(t)dt |. \label{e2}
\end{align}
Consequently, applying the inequalities (\ref{time}), (\ref{1e}) ,(\ref{step}) and (\ref{e2}), we have
\begin{align}
|R_{n}(\bar{t}_n)|
\leq K_6\int_{t_n}^{\bar{t}_n}(\tau_{n+1}^2t_{n}^{\delta-2}+\tau_{n+1}t_n^{\delta-2}){(\bar{t}_n-s)^{-\alpha}}ds\leq K (n+1)^{-(2-\alpha)}. 
\end{align}
\end{description}
Collecting all the above estimates completes the proof.
\end{proof}
Finally, we present the error estimate for the SOE discrete operator.
\begin{theorem} \label{FRu}
Under the regularity assumptions \eqref{regularity}, for any prescribed accuracy $\varepsilon > 0$ given in \eqref{SOE}, there exists a constant $K$, independent of $n$, $N$, and $\varepsilon$ , such that
    \begin{align}
        |{}^{FC}Ru(\bar{t}_n)|&= |{ }_0^{C}D_t^{\alpha,  \lambda} u(\bar{t}_n)-{ }_0^{FC}D_t^{\alpha,  \lambda} u(\bar{t}_n)|\leq 
        \begin{cases} 
        K(N^{-r(\delta+1-\alpha)}+\varepsilon),&n=0,\\
        K((n+1)^{-\min\{2-\alpha, r(1+\delta)\}}+\varepsilon), &n\geq 1.
        \end{cases}\label{Fhalf}
    \end{align}
\end{theorem}
\begin{proof}
    We decompose the truncation error into two parts as follows:
    \begin{align*}
        { }^{F C} R u\left(\bar{t}_n\right)=\underbrace{{ }_0^C D_t^{\alpha, \lambda} u\left(\bar{t}_n\right)-{ }_0^C \mathbb{D}_t^{\alpha, \lambda} u\left(\bar{t}_n\right)}_{{ }^{L C} R u\left(\bar{t}_n\right)}+\underbrace{{ }_0^C \mathbb{D}_t^{\alpha, \lambda} u\left(\bar{t}_n\right)-{ }_0^{F C} D_t^{\alpha, \lambda} u\left(\bar{t}_n\right)}_{\text { auxiliary error }} .
    \end{align*}
    The first term ${}^{LC}Ru(\bar{t}_n)$ can be directly bounded by Lemma~\ref{Ru}. It thus remains to estimate the auxiliary error.
From the definitions of our operator~\eqref{fastoperator} and the L1 operator~\eqref{L1operator}, the auxiliary error can be further bounded by a history term and a local term:
\begin{align*}
    |{}^C_0\mathbb{D}_t^{\alpha,  \lambda} u(\bar{t}_n)-{ }_0^{FC}D_t^{\alpha,  \lambda} u(\bar{t}_n)|
\leq  {}^{FC}R_hu(\bar{t}_n)+{}^{FC}R_lu(\bar{t}_n),
\end{align*}
where there exist constants $K_1$ and $K_2$, independent of $n$ and $\varepsilon$, such that
\begin{align*}
{}^{FC}R_h u(\bar{t}_n)
&\le K_1 \sum_{k=0}^{n-1}  
\int_{t_k}^{t_{k+1}} e^{\lambda s}
|L_{1,k+1} u(s)||
(\bar{t}_n - s)^{-\alpha-1}
- \sum_{l=1}^{N_{\mathrm{exp}}} w_l e^{-s_l(\bar{t}_n - s)}| \, ds, \\[1ex]
{}^{FC}R_l u(\bar{t}_n)
&\le K_2 \int_{t_n}^{\bar{t}_n}|
\frac{
e^{\lambda \bar{t}_n} L_{1,n+1} u(\bar{t}_n)
- e^{\lambda t_n} L_{1,n+1} u(t_n)
}{\tau_{n+1}/2}
- ( e^{\lambda s} L_{1,n+1} u(s) )'|
(\bar{t}_n - s)^{-\alpha} \, ds .
\end{align*}
For ${}^{FC}R_h u(\bar{t}_n)$, using the SOE approximation error bound \eqref{SOE}, we obtain
\begin{align}
    {}^{FC}R_h u(\bar{t}_n)\leq K_1 \varepsilon \sum_{k=0}^{n-1}  \int_{t_k}^{t_{k+1}} e^{\lambda s}|L_{1,k+1} u(s)| ds.
\end{align}
Invoking the regularity assumptions \eqref{regularity}, it follows that
\begin{align}
    {}^{FC}R_h u(\bar{t}_n)\leq K\varepsilon.
\end{align}
For ${}^{FC}R_l u(\bar{t}_n)$, applying Taylor expansion to $(e^{\lambda s} L_{1,n+1} u(s))'$, we obtain
\begin{align}
{}^{FC}R_lu(\bar{t}_n)
\leq K_2 \int_{t_{n}}^{\bar{t}_n} (\bar{t}_n-t_n)\max_{\theta\in[t_n, \bar{t}_n]} |(e^{\lambda \theta}L_{1, n+1}u(\theta))''|(\bar{t}_n-s)^{-\alpha}ds.
\end{align}
Invoking the regularity assumptions \eqref{regularity} yields
\begin{align}
{}^{FC}R_lu(\bar{t}_n)
    \leq K(\bar{t}_n^\delta \tau_{n+1}^{2-\alpha} + \bar{t}_n^{\delta-1} \tau_{n+1}^{2-\alpha})
\leq KN^{-(2-\alpha)}.\label{0e}
\end{align}
For the special case $n=0$, noting that ${}^{FC}R_h u(\bar{t}_0) = 0$, it follows from \eqref{0e} that
    \begin{align}
    {}^{FC}R_lu(\bar{t_0})\leq K (\bar{t_0}^{\delta}\tau_{1}^{2-\alpha}+\bar{t_0}^{\delta-1}\tau_{1}^{2-\alpha})
    \leq KN^{-r(1+\delta-\alpha)}.
    \end{align}
Combining the above estimate with \eqref{L1half} gives \eqref{Fhalf}. The proof is completed.
\end{proof}
\begin{rem}
    Note that \eqref{Fhalf} not only refines the result in \cite{cao2023}, where the error of ${}^{FC}Ru(t_n)$ was estimated by
$K(n^{-\min\{2-\alpha, r(1+\delta)\}} + \varepsilon)$, but also extends the corresponding error bound to the half-grid points.
\end{rem}
\section{Construction of the fully discrete scheme and the well-posedness}\label{section3}
In this section, based on the discrete approximation \eqref{fastoperator} of the Caputo tempered fractional derivative and the error analysis given in Theorem~\ref{FRu}, we construct a fully discrete scheme for  \eqref{Equation} in the next Sect.~\ref{fullscheme}. The unique solvability and stability of the proposed scheme are established in Sect.~\ref{sectionsolvable}. Furthermore, the convergence analysis and rigorous error estimates for the numerical solution are presented in Sect.~\ref{secconvergen}.

\subsection{Construction of the fully discrete scheme }
\label{fullscheme}
We first introduce some standard notations. Let $M$ and $N$ be two positive integers. Define the spatial and temporal step sizes by $h = L/M$ and $ t_n=T(n/{N})^r(0\leq n\leq N)$, and denote $\tau_n = t_n - t_{n-1}$. Let $x_i = ih$ for $0 \le i \le M$, and define the spatial and temporal grids $X_M=\{x_i|0\leq i \leq M\}$, $T_N=\{t_k|0\leq k\leq N\}$.  Define the mesh function as follows
\begin{align*}
u_i^n &= u(x_i, t_n), \quad 0\leq i \leq M, \quad 0\leq n\leq N, \\
f_i^{\bar{n}}&=f(x_i, \bar{t}_n),\quad 0\leq i \leq M,\quad  0\leq n\leq N-1. 
\end{align*}
We further introduce the following difference operators:
\begin{equation}\label{discretization}
\begin{aligned}
{\delta}_t u_i^{\bar{n}}&=\frac{u_i^{n+1}-u_i^n}{\tau_{n+1}}, \quad
{\delta}_x u_i^n=\frac{u_{i+1}^n-u_{i-1}^n}{2 h}, \quad {\delta}_x u_i^{\bar{n}} = \frac{{\delta}_x u_i^{n+1}+{\delta}_x u_i^n}{2} ,\\
\delta_x^2 u_i^n&=\frac{u^n_{i+1}-2u_i^n+u^n_{i-1}}{h^2},\quad \delta_x^2 u_i^{\bar{n}}=\frac{\delta_x^2 u_i^{n+1}+\delta_x^2 u_i^n}{2}. 
\end{aligned}
\end{equation}

Based on the SOE approximation (\ref{fastoperator}) for the Caputo fractional derivative and replacing the exact solution $u_i^n$ by its numerical approximation $U_i^n$, we obtain the following Crank–Nicolson-type fully discrete scheme:
\begin{align}\label{scheme}
\begin{cases}
\delta_t U_i^{\bar{n}} + {}_0^{FC}D_t^{\alpha, \lambda} U_i^{\bar{n}}
= \delta_x^2 U_i^{\bar{n}} - \delta_x U_i^{\bar{n}} + f_i^{\bar{n}},
& 1 \le i \le M-1,\quad 0 \le n \le N-1, \\
U_i^0 = \phi(x_i),
& 1 \le i \le M-1, \\
U_0^n = U_M^n = 0,
& 0 \le n \le N.
\end{cases}
\end{align}
Replacing the our approximation ${}_0^{FC}D$ by L1 evaluation scheme $_0^C\mathbb{D}$ \eqref{L1operator}, we obtain the L1 scheme of the following form
\begin{align}\label{L1scheme}
\begin{cases}
\delta_t U_i^{\bar{n}} + _0^C\mathbb{D}_t^{\alpha, \lambda} U_i^{\bar{n}}
= \delta_x^2 U_i^{\bar{n}} - \delta_x U_i^{\bar{n}} + f_i^{\bar{n}},
& 1 \le i \le M-1,\quad 0 \le n \le N-1, \\
U_i^0 = \phi(x_i),
& 1 \le i \le M-1, \\
U_0^n = U_M^n = 0,
& 0 \le n \le N.
\end{cases}
\end{align}
Comparing $_0^C\mathbb{D}$ in \eqref{L1operator} with ${}_0^{FC}D$ in \eqref{fastoperator}, 
the former requires all previous time-step values of $u$, whereas the latter only uses the history variables $U_{hist,i}(t_n)$, $i=1,\ldots,N_{exp}$, due to \eqref{Uequa}. 
Replacing $_0^C\mathbb{D}$ by ${}_0^{FC}D$ reduces the storage cost from $\mathcal{O}(N)$ to $\mathcal{O}(N_{exp})$ and the computational cost from $\mathcal{O}(N^2)$ to $\mathcal{O}(NN_{exp})$. 
Accordingly, the overall complexity decreases from $\mathcal{O}(MN^2)$ for the L1 scheme~\eqref{L1scheme} to $\mathcal{O}(MNN_{exp})$ for scheme~\eqref{scheme}, where $N_{exp}$ denotes the SOE approximation number defined in \eqref{SOE}. For fixed $\varepsilon$, $N_{exp}=\mathcal{O}(\log N)$ when $T\gg1$ and $N_{exp}=\mathcal{O}(\log^2 N)$ when $T\approx1$~\cite{jiang2017,zhu2019}. 
Hence, the proposed scheme significantly reduces storage and computational costs, especially for long-time simulations. Numerical comparisons are reported in Sect.~\ref{section4}.

\begin{rem}
    In contrast to~\cite{cao2020}, where the discretization is performed at integer time levels, we employ a half–grid-point discretization in time to account for the presence of the first-order time derivative term $u_t$. This treatment allows us to enhance the temporal accuracy to second order. Moreover, such an accuracy improvement is rigorously justified by the theoretical convergence analysis presented in Sect.~\ref{secconvergen} and is further confirmed by the numerical experiments reported in Sect.~\ref{section4}.
\end{rem}
\subsection{The unique solvability and stability of the fully discrete scheme} 
\label{sectionsolvable}
In this section, we analyze the unique solvability and stability of the proposed fully discrete scheme (\ref{scheme}) with respect to the initial conditions by employing the Fourier method (see, e.g., \cite{chen2007, Liu2018, zhao2021}).

Prior to stating the main theorem, we first provide an analysis of the scheme. Let
\begin{align*}
    \{U^n &= (U_0^n, \ldots, U_M^n) \mid 0 \le n \le N\},\\
    \{V^n &= (V_0^n, \ldots, V_M^n) \mid 0 \le n \le N\},
\end{align*}
denote the numerical solutions of scheme (\ref{scheme}) corresponding to two distinct initial conditions. Define the error \begin{align}
    \rho_j^n = U_j^n - V_j^n,\quad 0\leq j\leq M ,\quad 0\le n\le N. \label{expression}
\end{align}
By substituting \eqref{expression} into the difference scheme \eqref{scheme}, the resulting error satisfies
\begin{equation}\label{errorQ}
\left\{
\begin{aligned}
&
(-\tfrac{1}{2 h^2}+\tfrac{1}{4h}) \rho_{j+1}^{n+1}
+ (\tfrac{1}{\tau_{n+1}}+\tfrac{\tau_{n+1}^{-\alpha}}{2^{1-\alpha}\Gamma(2-\alpha)}+\tfrac{1}{h^2}) \rho_j^{n+1}
+ (-\tfrac{1}{2 h^2}-\tfrac{1}{4h}) \rho_{j-1}^{n+1}
 \\
&=
(\tfrac{1}{2 h^2}-\tfrac{1}{4h}) \rho_{j+1}^{n}
+ (\tfrac{1}{\tau_{n+1}}-\tfrac{1-2e^{-\lambda {\tau_{n+1}}/{2}}}{2^{1-\alpha}\Gamma(2-\alpha)}\tau_{n+1}^{-\alpha}-\tfrac{1}{h^2}) \rho_j^{n}
+ (\tfrac{1}{2 h^2}+\tfrac{1}{4h}) \rho_{j-1}^{n} \\
&\quad
+ \tfrac{1}{\Gamma(1-\alpha)}
(
(a_{0, n}-e^{-\lambda {\tau_{n+1}}/{2}} ({\tau_{n+1}}/{2})^{-\alpha}) \rho_j^n
+ \sum_{l=1}^{n-1} (a_{n-l, n}+b_{n-1-l, n}) \rho_j^l\\
&\quad+ (b_{n-1, n}+{e^{-\lambda \bar{t}_n}}{\bar{t}_n^{-\alpha}}) \rho_j^0
), \\
&\rho_0^n=\rho_M^n=0.
\end{aligned}
\right.
\end{equation}

Define the piecewise grid function 
\begin{align}
\rho^n(x)= \begin{cases}0,  & x_0 \leq x \leq \bar{x}_{0},  \\ \rho_j^n,  & \bar{x}_{j-1} \leq x \leq \bar{x}_{j},  \quad 1 \leq j \leq M-1,  \\ 0,  & \bar{x}_{M-1} \leq x \leq x_M . \end{cases}\notag
\end{align}
Since $\rho_j^n$ vanishes at the boundaries, it can be extended periodically, allowing us to apply the discrete Fourier transform (DFT). Denote the DFT of $\rho^n$ by
\begin{equation}
\hat{\rho}^n[k]=\sum_{j=0}^{M-1} e^{-i \frac{2 \pi}{M} j k} \rho^n_j, \quad 0\le k\le M-1, \notag
\end{equation}
with the inverse transform given by
\[
\rho^n_j = \frac{1}{M} \sum_{k=0}^{M-1} \hat{\rho}^n[k] e^{i \frac{2\pi}{M} jk}.
\]
By Parseval’s identity, we have
$$
\int_0^L|\rho^n(x)|^2 d x=h\sum_{j=0}^{M-1}|\rho^n_j|^2=Mh\sum_{k=0}^{M-1}|\hat{\rho}^n[k]|^2 =L\sum_{k=0}^{M-1}|\hat{\rho}^n[k]|^2. 
$$
Consequently, we obtain
\begin{align}
    \|\rho^n(x)\|_{L_2}^2=(\sum_{j=0}^{M-1} h|\rho_j^n|^2)=(\int_0^L|\rho^n(x)|^2 d x)=L\sum_{k=0}^{M-1}|\hat{\rho}^n[k]|^2.\label{rhoequal}
\end{align}

We multiply both sides of equation (\ref{errorQ}) by $e^{-ijh\beta}$  with $\beta = 2\pi k/M$, and sum over $j$. By employing the identities $e^{ih\beta}+e^{-ih\beta}=2\cos(h\beta)$ and $e^{ih\beta}-e^{-ih\beta}=2i\sin(h\beta)$, we arrive at the following system:
\begin{equation}\label{d0}
\left\{
\begin{aligned}
    &(-\frac{1}{h^2}\cos(h\beta)+\frac{1}{2h}i\sin(h\beta)+(\frac{1}{\tau_{n+1}}+\frac{\tau_{n+1}^{-\alpha}}{2^{1-\alpha}\Gamma(2-\alpha)}+\frac{1}{h^2}) )\hat{\rho}^{n+1}[k] \\
    &=(\frac{1}{h^2}\cos(h\beta)-\frac{1}{2h}i\sin(h\beta)+(\frac{1}{\tau_{n+1}}-\frac{\tau_{n+1}^{-\alpha}}{2^{1-\alpha}\Gamma(2-\alpha)}-\frac{1}{h^2}) )\hat{\rho}^{n}[k] \\
    &\quad +\frac{1}{\Gamma(1-\alpha)}((a_{0, n} +\frac{e^{-\lambda {\tau_{n+1}}/{2}}\tau_{n+1}^{-\alpha}}{2^{-\alpha}(1-\alpha)}-e^{-\lambda {\tau_{n+1}}/{2}} ({\tau_{n+1}}/{2})^{-\alpha})\hat{\rho}^{n}[k]\\
    &\quad +\sum\limits_{l=1}^{n-1} (a_{n-l, n}+b_{n-1-l, n}) \hat{\rho}^{l}[k]+(b_{n-1, n}+{e^{-\lambda \bar{t}_n}}{\bar{t}_n^{-\alpha}}) \hat{\rho}^{0}[k]).
\end{aligned}
\right.
\end{equation}
\begin{lemma}\label{dnd0}
    Suppose that $\hat{\rho}_{n}[k]$ $ (1\le n\le N)$ satisfy (\ref{d0}).  Then 
    \begin{align}
    |\hat{\rho}^{n}[k]|\leq |\hat{\rho}^{0}[k]|,\quad 1\le n\le N.\label{rho1}
    \end{align}
\end{lemma}
\begin{proof}
Use induction on $n$. The case $n=1$ of (\ref{rho1}) is
\begin{align}
    |\hat{\rho}_{1}[k]|=\left|\frac{\frac{1}{h^2}\cos(h\beta)-\frac{1}{2h}i\sin(h\beta)+(\frac{1}{\tau_{1}}-\frac{1-2e^{-\lambda {\tau_{1}}/{2}}}{2^{1-\alpha}\Gamma(2-\alpha){\tau_{1}}^\alpha}-\frac{1}{h^2}) }{-\frac{1}{h^2}\cos(h\beta)+\frac{1}{2h}i\sin(h\beta)+(\frac{1}{\tau_{1}}+\frac{1}{2^{1-\alpha}\Gamma(2-\alpha){\tau_{1}}^\alpha}+\frac{1}{h^2}) }\right||\hat{\rho}^{0}[k]| \leq |\hat{\rho}^{0}[k]|.
\end{align}
Fix $n\in \{1,2,\dots,N-1\}$. Assume that (\ref{rho1}) is valid for \(m=1,2,\dots,n\). Then (\ref{lambda}), (\ref{ab})
and the inductive hypothesis yield
\begin{align}
  a_{0,n}+\sum_{l=1}^{n-1}(a_{l, n}+b_{l, n})+b_{n-1,n}  
  &=\sum_{l=0}^{n-1}\alpha \sum_{i=1}^{N_{exp}}w_i\int_{t_{n-l-1}}^{t_{n-l}}e^{-(\lambda+s_i)(\bar{t}_n-s)}ds.\notag
\end{align}
Combining the SOE approximation \eqref{SOE} and the above expression yields
\begin{align}
  a_{0,n}+\sum_{l=1}^{n-1}(a_{l, n}+b_{l, n})+b_{n-1,n}  
&=\alpha\int_{{\tau_{n+1}}/{2}}^{\bar{t}_n}{e^{-\lambda s}}{s^{-1-\alpha}}ds\notag\\
  &\leq {e^{-\lambda {\tau_{n+1}}/{2}}}{({\tau_{n+1}}/{2})^{-\alpha}}-{e^{-\lambda \bar{t}_n}}{\bar{t}_n^{-\alpha}}.\label{aplusb}
\end{align}
Applying the inequality (\ref{aplusb}) to (\ref{d0}), we get
\begin{align}
    &|-\frac{1}{h^2}\cos(h\beta)+\frac{1}{2h}i\sin(h\beta)+(\frac{1}{\tau_{n+1}}+\frac{\tau_{n+1}^{-\alpha}}{2^{1-\alpha}\Gamma(2-\alpha)}+\frac{1}{h^2}) ||\hat{\rho}^{n+1}[k]|\notag \\
    \leq&|\frac{1}{h^2}\cos(h\beta)-\frac{1}{2h}i\sin(h\beta)+(\frac{1}{\tau_{n+1}}-\frac{\tau_{n+1}^{-\alpha}}{2^{1-\alpha}\Gamma(2-\alpha)}-\frac{1}{h^2}) ||\hat{\rho}^{n}[k]| +\frac{e^{-\lambda {\tau_{n+1}}/{2}}\tau_{n+1}^{-\alpha}}{2^{-\alpha}\Gamma(2-\alpha)}|\hat{\rho}^{0}[k]|.\label{need1}
\end{align}
Define
\begin{align*}
A &= \frac{1-\cos(h\beta)}{h^2},\quad B=\frac{\tau_{n+1}^{-\alpha}}{2^{1-\alpha}\Gamma(2-\alpha)},\quad C=2Be^{-\lambda{\tau_{n+1}}/{2}},\\
a&=-A+\frac{1}{\tau_{n+1}}-B,\quad b=A+\frac{1}{\tau_{n+1}}+B,\quad d=\frac{\sin(h\beta)}{2h}.
\end{align*}
To prove that $|\hat{\rho}^{n+1}[k]|\leq |\hat{\rho}^{0}[k]|$ by \eqref{need1}, it suffices to establish the following inequality:
\begin{align*}
& \sqrt{a^2 + d^2} + C \le \sqrt{b^2 + d^2}.
\end{align*}
Noticing that $d^2 \leq {A}/{2}$ and $(A+B)^2-B^2e^{-\lambda\tau_{n+1}}\geq A^2+2AB$, it is sufficient to verify
\begin{align}
    {Be^{-\lambda\tau_{n+1}}\tau_{n+1}^2}\leq 4({1-B^2e^{-\lambda\tau_{n+1}}\tau_{n+1}^2}).\label{need2}
\end{align}
Using the bound \(1/2<\Gamma(2-\alpha)<1\) and straightforward estimates, inequality (\ref{need2}) can be reduced to the simpler condition
\[
\tau_{n+1}^{2-2\alpha}<1/3,
\]
which is easily satisfied in typical numerical simulations. Consequently, (\ref{need1}) and (\ref{need2}) yield \(|\hat{\rho}^{n+1}[k]|\leq |\hat{\rho}^{0}[k]|\). Thus we
have proved (\ref{rho1}) for $m = n+1$. By the principle of induction, the lemma is proved.
\end{proof}

Based on the above analysis and Lemma \ref{dnd0}, we are now in a position to present the unique solvability and stability for the difference scheme (\ref{scheme}) in Theorems \ref{stable}.
\begin{theorem}\label{stable}
The finite difference scheme (\ref{scheme}) is uniquely solvable and stable. Let $\{U^n = (U_0^n, \ldots, U_M^n) \mid 0 \le n \le N\}$ and $\{V^n = (V_0^n, \ldots, V_M^n) \mid 0 \le n \le N\}$ denote the solutions of the scheme (\ref{scheme}) corresponding to two different initial values. Then
\begin{align}
\|U^n-V^n\|_{L_2} \leq \|U^0-V^0\|_{L_2},\quad 1\le n\le N. 
\end{align}
\end{theorem}
\begin{proof}
First, we prove the unique solvability of the proposed difference scheme. Let 
        $$
        U^n=(U_0^n, U_1^n, \ldots, U_M^n). 
        $$
        From (\ref{scheme}), the initial value $U^0$ is known. Assume that the first $n$ layers, $U^0, U^1, \ldots, U^n$, are uniquely determined. Then, the scheme yields a linear system for $U^{n+1}$. To prove uniqueness, it suffices to show that the corresponding homogeneous system 
        \begin{equation}
            \left\{
        \begin{aligned}
        &\eta {U_i^{n+1}}=\frac{1}{2}\delta_x^2U_i^{n+1}-\frac{1}{2}{\delta}_xU_i^{n+1}, \quad1\leq i\leq M-1,\notag\\
        &U_0^n=U_M^n=0 ,\label{unique}
        \end{aligned}
        \right.
        \end{equation}
        has only the zero solution, where $\eta = \frac{1}{\tau_{n+1}} + \frac{\tau_{n+1}^{-\alpha}}{2^{1-\alpha}\Gamma(2-\alpha)}$.
Let
\[
\mathbf{U}^{n+1}
= \big( U_1^{n+1}, U_2^{n+1}, \ldots, U_{M-1}^{n+1} \big)^{\mathsf T},
\]
and define the coefficient matrix $A \in \mathbb{R}^{(M-1)\times(M-1)}$ by
\[
A =
\begin{pmatrix}
\eta+\dfrac{1}{h^2} & \dfrac{1}{4h}-\dfrac{1}{2h^2} & 0 & \cdots & 0 \\
-\dfrac{1}{4h}-\dfrac{1}{2h^2} & \eta+\dfrac{1}{h^2} & \dfrac{1}{4h}-\dfrac{1}{2h^2} & \cdots & 0 \\
0 & -\dfrac{1}{4h}-\dfrac{1}{2h^2} & \eta+\dfrac{1}{h^2} & \cdots & 0 \\
\vdots & \vdots & \ddots & \ddots & \vdots \\
0 & 0 & \cdots & -\dfrac{1}{4h}-\dfrac{1}{2h^2} & \eta+\dfrac{1}{h^2}
\end{pmatrix}.
\]
Then the discrete system (\ref{unique}) can be written compactly as
\[
A \, \mathbf{U}^{n+1} = \mathbf{0}.
\]
Since $A$ is tridiagonal, its eigenvalues are given by \cite{gregory1969} as
$$
\lambda_k=\eta+\frac{1}{h^2}+2\sqrt{(\frac{1}{4h}-\frac{1}{2h^2})(-\frac{1}{4h}-\frac{1}{2h^2})}\cos(\frac{k\pi}{M}),\quad 1\le k\le M-1.
$$
Since all eigenvalues satisfy $\lambda_k \ge \eta > 0$, the matrix $A$ is invertible. Consequently, (\ref{unique}) admits only the trivial solution.By induction, the uniqueness of the solution for all $n$ follows.

    We proceed to the stability analysis of the scheme. Using Lemma~\ref{dnd0} together with~(\ref{rhoequal}), we obtain
    \begin{equation*}
\begin{aligned}
\|U^n-V^n\|_{L_2}^2 =L\sum_{k=0}^{M-1}|\hat{\rho}^n[k]|^2 \leq L\sum_{k=0}^{M-1}|\hat{\rho}^0[k]|^2=\|U^0-V^0\|_{L_2}^2.
\end{aligned}
\end{equation*}
Taking the square root on both sides yields the desired estimate and completes the proof.
\end{proof}

\subsection{Convergence of the fully discrete scheme}\label{secconvergen}
We now turn to the convergence analysis of the proposed difference scheme. To this end, we need some preparations.
\begin{lemma} \label{second}
Under the regularity assumptions \eqref{regularity}, 
there exists a positive constant $K$, independent of $h$, $N$, and $n$, such that
\begin{equation}\label{spaceerror}
    \begin{aligned}
    |R_t^{i, \bar{n}}| &= |u_t(x_i,  \bar{t}_n) -{\delta}_t u_i^{\bar{n}}| \leq \begin{cases} 
        KN^{-r(\delta-1)},&n=0,\\
        K(n+1)^{-2}, &n\geq 1,
        \end{cases}\\
    |R_x^{i, \bar{n}}|&=|u_x(x_i, \bar{t}_n)-{\delta}_x u_i^{\bar{n}}|\leq \begin{cases} 
        K(N^{-r\delta}+h^2),&n=0,\\
        K((n+1)^{-2}+h^2), &n\geq 1,
        \end{cases}\\
     |R_{xx}^{i, \bar{n}}|&=|u_{xx}(x_i, \bar{t}_n)-\delta_x^2u_i^{\bar{n}}|\leq \begin{cases} 
        K(N^{-r\delta}+h^2),&n=0,\\
        K((n+1)^{-2}+h^2), &n\geq 1.
        \end{cases}
    \end{aligned}
\end{equation}
\end{lemma}
\begin{proof}
We divide the proof into three parts, corresponding to the temporal, first spatial, and second spatial truncation errors.

\medskip
\noindent\textbf{Estimate for $R_t^{i,\bar n}$.}
For $n \ge 1$, Taylor’s expansion with integral remainder implies that there exists $\xi \in (t_n, t_{n+1})$ such that
\[
u_t(x_i,\bar{t}_n) - \delta_t u_i^{\bar{n}}
= \int_{\bar{t}_n}^{\xi} u_{ttt}(x_i,t) (\xi - t) \, dt.
\]
By the regularity assumption~\eqref{regularity} and \eqref{step}, we have
\begin{align*}
|u_t(x_i,\bar{t}_n) - \delta_t u_i^{\bar{n}}|
\le K t_n^{\delta-3} \tau_{n+1}^2
\le K (n+1)^{-2}, \quad n \ge 1.
\end{align*}
For the initial step $n=0$, a direct Taylor expansion together with~\eqref{regularity} yields
\begin{align*}
|u_t(x_i,\bar{t}_0) - \delta_t u_i^{\bar{0}}|
&= | \frac{1}{2\tau_1} \int_{\bar{t}_0}^{t_1} u_{ttt}(x_i,t)(t_1^2 - 2t_1 t) \, dt
      + \frac{1}{2\tau_1} \int_{t_0}^{t_1} u_{ttt}(x_i,t) t^2 \, dt | \le K N^{-r(\delta-1)}.
\end{align*}
Hence, the temporal truncation error satisfies
\[
|R_t^{i,\bar{n}}|
\le
\begin{cases}
K N^{-r(\delta-1)}, & n=0, \\
K (n+1)^{-2}, & n \ge 1.
\end{cases}
\]

\medskip
\noindent\textbf{Estimate for $R_x^{i,\bar{n}}$.}
By applying Taylor expansion in time and integration by parts, we have
\begin{align}
|u_x(x_i,\bar{t}_n) - \delta_x u_i^{\bar{n}}|
&\le
| \int_{\bar{t}_n}^{t_{n+1}} u_{xtt}(x_i,t)(t_{n+1}-t)\,dt
+ \int_{t_n}^{\bar{t}_n} u_{xtt}(x_i,t)(t_n - t)\,dt |
+ O(h^2) \notag \\
&\le C ( t_{n+1}^{\delta-2} \tau_{n+1}^2
+ \int_{t_n}^{\bar{t}_n} t^{\delta-2}(t-t_n)\,dt ) + O(h^2). \label{yijiechashang}
\end{align}
For $n=0$, estimate~\eqref{yijiechashang} together with~\eqref{step} directly yields
\[
|u_x(x_i,\bar{t}_0) - \delta_x u_i^{\bar{0}}|
\le K N^{-r\delta} + O(h^2).
\]
For $n \ge 1$, using the inequality $t_{n+1} \le 2^r t_n$ and estimate~\eqref{yijiechashang}, we deduce
\[
|u_x(x_i,\bar{t}_n) - \delta_x u_i^{\bar{n}}|
\le K (n+1)^{-2} + O(h^2).
\]
Hence, the first spatial truncation error satisfies
\[
|R_x^{i,\bar{n}}|
\le
\begin{cases}
K (N^{-r\delta} + h^2), & n=0, \\
K ((n+1)^{-2} + h^2), & n \ge 1.
\end{cases}
\]

\medskip
\noindent\textbf{Estimate for $R_{xx}^{i,\bar{n}}$.}
The analysis for the second spatial derivative is analogous. Taylor expansion and integration by parts yield
\begin{align*}
|u_{xx}(x_i,\bar{t}_n) - \delta_x^2 u_i^{\bar{n}}|
&\le
| \frac{1}{2} \int_{\bar{t}_n}^{t_{n+1}} u_{xxtt}(x_i,t)(t_{n+1}-t)\,dt
+ \frac{1}{2} \int_{t_n}^{\bar{t}_n} u_{xxtt}(x_i,t)(t_n - t)\,dt |
+ O(h^2) \\
&\le
C ( t_{n+1}^{\delta-2} \tau_{n+1}^2
+ \int_{t_n}^{\bar{t}_n} t^{\delta-2}(t-t_n)\,dt )
+ O(h^2),
\end{align*}
which leads to the same bounds as (\ref{yijiechashang}) for $R_x^{i,\bar{n}}$. Therefore,
\[
|R_{xx}^{i,\bar{n}}|
\le
\begin{cases}
K (N^{-r\delta} + h^2), & n=0, \\
K ((n+1)^{-2} + h^2), & n \ge 1.
\end{cases}
\]
Combining the above estimates completes the proof.
\end{proof}
With the above lemma, we are now in a position to derive the error equation and carry out the convergence analysis.
Let $u_i^n$ and $U_i^n$ be the exact and numerical solutions of 
\eqref{Equation} and the difference scheme \eqref{scheme} at $(x_i, t_n)$, respectively. 
Define the error function by
\[
\varepsilon_i^n = u_i^n - U_i^n.
\]
Then $\varepsilon_i^n$ satisfies the following error equation:
\begin{equation}
    \left\{
\begin{aligned}
&
(-\tfrac{1}{2 h^2}+\tfrac{1}{4h}) \varepsilon_{j+1}^{n+1}
+ (\tfrac{1}{\tau_{n+1}}+\tfrac{\tau_{n+1}^{-\alpha}}{2^{1-\alpha}\Gamma(2-\alpha)}+\tfrac{1}{h^2}) \varepsilon_j^{n+1}
+ (-\tfrac{1}{2 h^2}-\tfrac{1}{4h}) \varepsilon_{j-1}^{n+1}
 \\
&=
(\tfrac{1}{2 h^2}-\tfrac{1}{4h}) \varepsilon_{j+1}^{n}
+ (\tfrac{1}{\tau_{n+1}}-\tfrac{1-2e^{-\lambda {\tau_{n+1}}/{2}}}{2^{1-\alpha}\Gamma(2-\alpha)}\tau_{n+1}^{-\alpha}-\tfrac{1}{h^2}) \varepsilon_j^{n}
+ (\tfrac{1}{2 h^2}+\tfrac{1}{4h}) \varepsilon_{j-1}^{n} \\
&\quad
+ \tfrac{1}{\Gamma(1-\alpha)}
(
(a_{0, n}-e^{-\lambda {\tau_{n+1}}/{2}} ({\tau_{n+1}}/{2})^{-\alpha}) \varepsilon_j^n
+ \sum_{l=1}^{n-1} (a_{n-l, n}+b_{n-1-l, n}) \varepsilon_j^l\\
&\quad+ (b_{n-1, n}+{e^{-\lambda \bar{t}_n}}{\bar{t}_n^{-\alpha}}) \varepsilon_j^0
)+R_j^{\bar{n}}, \label{converge1}
\end{aligned}
\right.
\end{equation}
where $R_j^{\bar{n}}$ denotes the local truncation error at $(x_j, \bar{t}_n)$. 
From Theorem~\ref{FRu} and Lemma~\ref{second}, it follows that
\begin{align}\label{errorall}
    |R_j^{\bar{n}}|&=|{}^{FC}R^{\bar{n}}+R_t^{i, \bar{n}}+R_x^{i, \bar{n}}+R_{xx}^{i, \bar{n}}|\leq \begin{cases}
K (N^{-r(\delta-1)} + h^2), & n=0, \\
K ((n+1)^{-(2-\alpha)} + h^2), & n \ge 1.
\end{cases}
\end{align}
Following the same procedure as in the stability analysis, we define the grid functions:
 \begin{equation}\notag
\varepsilon^n(x)=\left\{\begin{array}{l}
0,  \quad x_0 \leq x \leq \bar{x}_0,  \\
\varepsilon_j^n,  \quad \bar{x}_{j-1} \leq x \leq \bar{x}_j,  \quad 1 \leq j \leq M-1,  \\
0,  \quad \bar{x}_{M-1} \leq x \leq x_M , 
\end{array}\right.  
\end{equation}
\begin{equation}\notag
R^{\bar{n}}(x)=\left\{\begin{array}{l}
0,  \quad x_0 \leq x \leq \bar{x}_0,  \\
R^{\bar{n}}_j,  \quad \bar{x}_{j-1} \leq x \leq \bar{x}_j,  \quad 1 \leq j \leq M-1,  \\
0,  \quad \bar{x}_{M-1} \leq x \leq x_M . 
\end{array}\right. 
\end{equation}
Let $\hat{\varepsilon}^n[k]$ and $\hat{R}^{\bar{n}}[k]$ denote the discrete Fourier transforms of $\varepsilon^n(x)$ and $R^{\bar{n}}(x)$, respectively.
\begin{align}\label{suppose}
\hat{\varepsilon}^n[k]&=\sum_{j=1}^{M-1} e^{-i \frac{2 \pi}{M} j k} \varepsilon^n_j, \quad 
\hat{R}^{\bar n}[k]=\sum_{j=1}^{M-1} e^{-i \frac{2 \pi}{M} j k} R^{\bar n}_j, \quad 0\le k\le M-1 . 
\end{align}
Similar to (\ref{rhoequal}), Parseval’s equality then gives
\begin{align}
    \|\varepsilon^n(x)\|_{L_2}^2=L\sum_{k=0}^{M-1}|\hat{\varepsilon}^n[k]|^2,\quad \|R^{\bar{n}}(x)\|_{L_2}^2=L\sum_{k=0}^{M-1}|\hat{R}^{\bar{n}}[k]|^2.\label{parseval1}
\end{align}
Using \eqref{suppose}, multiplying both sides of \eqref{converge1} by $e^{-ijh\beta}$ with $\beta = 2\pi k/M$, and summing over $j$, we obtain
\begin{equation}\label{conLemma}
\left\{
\begin{aligned}
    &(-\frac{1}{h^2}\cos(h\beta)+\frac{1}{2h}i\sin(h\beta)+(\frac{1}{\tau_{n+1}}+\frac{\tau_{n+1}^{-\alpha}}{2^{1-\alpha}\Gamma(2-\alpha)}+\frac{1}{h^2}) )\hat{\varepsilon}^{n+1}[k] \\
    &=(\frac{1}{h^2}\cos(h\beta)-\frac{1}{2h}i\sin(h\beta)+(\frac{1}{\tau_{n+1}}-\frac{\tau_{n+1}^{-\alpha}}{2^{1-\alpha}\Gamma(2-\alpha)}-\frac{1}{h^2}) )\hat{\varepsilon}^{n}[k] \\
    &\quad +\frac{1}{\Gamma(1-\alpha)}((a_{0, n} +\frac{e^{-\lambda {\tau_{n+1}}/{2}}\tau_{n+1}^{-\alpha}}{2^{-\alpha}(1-\alpha)}-e^{-\lambda {\tau_{n+1}}/{2}} ({\tau_{n+1}}/{2})^{-\alpha})\hat{\varepsilon}^{n}[k]\\
    &\quad +\sum\limits_{l=1}^{n-1} (a_{n-l, n}+b_{n-1-l, n}) \hat{\varepsilon}^{l}[k]+(b_{n-1, n}+{e^{-\lambda \bar{t}_n}}{\bar{t}_n^{-\alpha}}) \hat{\varepsilon}^{0}[k])+\hat{R}^{\bar{n}}[k].
\end{aligned}
\right.
\end{equation}

 We now state the convergence theorem for the fully discrete scheme (\ref{scheme}).
\begin{theorem}\label{converge}
Under the regularity assumption~\eqref{regularity} and for
$r \ge \max\{ 2/(\delta-1),\, 3 \}$,
let $u(x,t)$ and $\{U^n = (U_0^n, \ldots, U_M^n)\}_{n=0}^N$
 be the solution of the problem~\eqref{Equation} and the scheme~\eqref{scheme}, respectively.
Then there exists a constant $C>0$, independent of $h$, $N$, and $\varepsilon$, such that
\begin{align}
\|u^n - U^n\|_{L^2} \le C \big( N^{-2} + h^2 + \varepsilon \big).
\label{2-order}
\end{align}
\end{theorem}

    \begin{proof}
    We first establish the following auxiliary result. Suppose that $\hat{\varepsilon}^n[k]$ $(1 \le n \le N)$ and $\hat{R}^{\bar{n}}[k]$ $(0 \le n \le N-1)$ satisfy \eqref{conLemma}. Then the following estimate holds:
\begin{align}
|\hat{\varepsilon}^n[k]| \le (1+\tau_1)^{n^{,r-2} N^{2}}  |\hat{R}^{\bar 0}[k]|,
\qquad 1 \le n \le N. \label{epsilon11}
\end{align}
We prove \eqref{epsilon11} by induction on $n$.
Since the boundary conditions imply $\varepsilon_j^0 = 0$, it follows from \eqref{suppose} that $\hat{\varepsilon}^0[k] = 0$. In view of condition \eqref{conLemma}, the case $n=1$ in \eqref{epsilon11} yields
    \begin{align}
    |\hat{\varepsilon}^{1}[k]|
    =|\hat{R}^{\bar 0}[k]|({( \frac{1}{\tau_1}+\frac{1}{2\Gamma(2-\alpha)(\frac{\tau_{1}}{2})^\alpha}+\frac{1-\cos(h\beta)}{h^2})^2+(\frac{\sin(h\beta)}{2h})^2})^{-\tfrac{1}{2}}\leq |\hat{R}^{\bar 0}[k]|.\notag
\end{align}
Fix $n\in\{1,2,\dots,N-1\}$. Assume that (\ref{epsilon11}) is valid for $m=1,2,\dots,n$. From \eqref{errorall} and the definition of the discrete Fourier transform \cite{chen2007,Liu2018, zhao2021}, we have
\[
|\hat{R}^{\bar n}[k]| \le {N^2}{(n+1)^{\alpha-2}} |\hat{R}^{\bar 0}[k]|.
\]
Using \eqref{conLemma} and following the same argument as in Lemma~\ref{dnd0}, we derive
\begin{align}
|\hat{\varepsilon}^{n+1}[k]|
&\le (1+\tau_1)^{n^{r-2}N^2} |\hat{R}^{\bar 0}[k]|
+ \tau_1 {N^2}{(n+1)^{\alpha-2}} |\hat{R}^{\bar 0}[k]| \notag \\
&\le ( (1+\tau_1)^{n^{r-2}N^2} + (1+\tau_1)^{{N^2}{(n+1)^{\alpha-2}}} - 1 ) |\hat{R}^{\bar 0}[k]|.\label{epsilon22}
\end{align}
Let $g = (1+\tau_1)^{{N^2}{(n+1)^{\alpha-2}}}$. 
Since $r \ge 3$ and $g\ge1$, a direct calculation then shows that
\begin{align*}
    g^{(n+1)^{r-\alpha}}-g^{n^{r-2}(n+1)^{2-\alpha}}-g+1\ge 0.
\end{align*}
Substituting this expression for $g$ into \eqref{epsilon22} yields
\[
|\hat{\varepsilon}^{n+1}[k]| \le (1+\tau_1)^{(n+1)^{r-2}N^2} |\hat{R}^{\bar 0}[k]|.
\]
Thus, inequality~(\ref{epsilon11}) holds for $m = n + 1$. By induction, inequality~(\ref{epsilon11}) holds for all $n$.
Using~(\ref{epsilon11}), we obtain
\begin{align}
    |\hat{\varepsilon}^n[k]|\leq (1+\tau_1)^{n^{r-2}N^{2}}|\hat{R}^{\bar 0}[k]|\leq (1+\frac{T}{N^r})^{N^r}|\hat{R}^{\bar 0}[k]|\leq e^T|\hat{R}^{\bar 0}[k]|.\notag
\end{align}
    Applying Parseval’s identity~(\ref{parseval1}) and setting $C = e^{T}$, we obtain
    \begin{equation}
\begin{aligned}
\|u^n-U^n\|_{L_2}^2 =L\sum_{k=0}^{M-1}|\hat{\varepsilon}^n[k]|^2\leq C^2L \sum_{k=0}^{M-1}|\hat{R}^{\bar 0}[k]|^2=C^2\|R^{\bar 0}\|_{L_2}^2\leq C^2KL(N^{-2}+h^2+\varepsilon)^2. 
\end{aligned}
\end{equation}
Taking the square root on both sides yields the desired estimate and completes the proof.  
\end{proof}
\begin{rem}
    The convergence order established for scheme~\eqref{scheme} is second order in both time and space. This result is in contrast to the first-order temporal accuracy reported in~\cite{cao2020} and the 
$1+\alpha$ order temporal accuracy obtained in~\cite{qiao2022}. In this respect, the present work provides a clear improvement over the existing results.
    \end{rem}
\section{Numerical Experiments}\label{section4}
 In this section, we apply the proposed fully discrete scheme \eqref{scheme} to \eqref{Equation}, where three cases for $f(x,t)$ and $\phi(x)$ given in \eqref{Equation} are considered, and the corresponding numerical results are presented to verify the theoretical results presented in Theorems \ref{FRu} and \ref{converge}. In addition, in all numerical experiments, we choose $r = 3$ in \eqref{graded} and $\varepsilon = 10^{-10}$ in \eqref{SOE}. All computations are carried out in MATLAB (R2023b) on a computer equipped with an Intel(R) Core(TM) i5 CPU at 2.40 GHz and 8 GB RAM.


\noindent\textbf{\underline{Case 1}}

We consider \eqref{Equation} with $\Omega=(0,1)$, $T=2$, $\phi(x)=x^2(1-x)^2$ and 
\begin{align}
f(x,t) = & ( -\lambda (t^\delta + 1) + \delta t^{\delta - 1} + \frac{\Gamma(\delta+1)}{\Gamma(\delta-\alpha+1)} t^{\delta-\alpha} ) e^{-\lambda t} x^2(1-x)^2 \notag\\
& - ( (12x^2 - 12x + 2) - (4x^3 - 6x^2 + 2x) ) (t^\delta + 1) e^{-\lambda t}.\notag
\end{align}
The exact solution is then given by $u(x,t) = e^{-\lambda t} (t^\delta + 1) x^2(1-x)^2$ (see e.g. \cite{Alikhanov2024_1,cao2020,chen2018,dwivedi2024,stynes2017}), which satisfies the regularity assumptions (\ref{regularity}).  For purposes of comparison, the results obtained by using the L1 scheme \eqref{L1scheme} are also presented.
\begin{table}[htbp]
\centering
\setlength{\tabcolsep}{5pt}   
\caption{A comparison of maximum $L^2$ error and convergence order in time with $\lambda=1$, $\delta=1.8$ and $M=2000$.}
\label{table:errors1t}
\begin{tabular}{c c c c c c c c c}
\toprule
& \multicolumn{4}{c}{$\alpha=0.25$} 
& \multicolumn{4}{c}{$\alpha=0.5$} \\
\cmidrule(lr){2-5} \cmidrule(lr){6-9}
 & \multicolumn{2}{c}{Our scheme (\ref{scheme})} & \multicolumn{2}{c}{L1 scheme (\ref{L1scheme})}   & \multicolumn{2}{c}{Our scheme (\ref{scheme})} & \multicolumn{2}{c}{L1 scheme (\ref{L1scheme})} \\
\cmidrule(lr){2-3} \cmidrule(lr){4-5} \cmidrule(lr){6-7} \cmidrule(lr){8-9}
\multirow{1}{*}{$N$} & Error & Order & Error & Order & Error & Order & Error & Order \\
\midrule
16
& 2.6375e-05  &  --  & 2.6358e-05 &  --  & 2.6521e-05 & --   & 2.6504e-05  & -- \\
32
& 6.7577e-06 &  2.0588 & 6.7844e-06 & 2.0519 & 6.9334e-06 & 2.0283   &  6.9753e-06 & 2.0116 \\
64
& 1.6808e-06  & 2.0540 & 1.6959e-06  & 2.0466   & 1.7685e-06 & 2.0168   & 1.7951e-06 & 2.0037 \\
\bottomrule
\end{tabular}
\end{table}
\begin{table}[htbp]
\centering
\setlength{\tabcolsep}{5pt}   
\caption{A comparison of  maximum $L^2$ error and convergence order in space with $\lambda=1$, $\delta=1.8$ and $N=500$.}
\label{table:errors1s}
\begin{tabular}{c c c c c c c c c}
\toprule
& \multicolumn{4}{c}{$\alpha=0.25$} 
& \multicolumn{4}{c}{$\alpha=0.5$} \\
\cmidrule(lr){2-5} \cmidrule(lr){6-9}
 & \multicolumn{2}{c}{Our scheme (\ref{scheme})} & \multicolumn{2}{c}{L1 scheme (\ref{L1scheme})}   & \multicolumn{2}{c}{Our scheme (\ref{scheme})} & \multicolumn{2}{c}{L1 scheme (\ref{L1scheme})} \\
\cmidrule(lr){2-3} \cmidrule(lr){4-5} \cmidrule(lr){6-7} \cmidrule(lr){8-9}
\multirow{1}{*}{$M$} & Error & Order & Error & Order & Error & Order & Error & Order \\
\midrule
20
& 3.3885e-01 & -- 
& 3.3885e-01 & -- 
& 3.2922e-01 & -- 
& 3.2922e-01 & -- \\
40
& 8.4711e-02 & 2.0000 
& 8.4711e-02 & 2.0000 
& 8.2276e-02 & 2.0000 
& 8.2282e-02 & 2.0000 \\
80
& 2.1138e-02 & 2.0030 
& 2.1138e-02 & 2.0030 
& 2.0503e-02 & 2.0040 
& 2.0510e-02 & 2.0040 \\
\bottomrule
\end{tabular}
\end{table}
\begin{enumerate}
    \item In Table~\ref{table:errors1t} and Table~\ref{table:errors1s}, we tabulate the the maximum $L^2$ errors and the corresponding convergence orders of our proposed scheme \eqref{scheme} and L1 scheme \eqref{L1scheme}. It is observed that both schemes produce essentially identical errors and achieve second-order accuracy in both time and space, which are in perfect agreement with the Theorem \ref{converge}. 
    \item Moreover, we plot in Figure~\ref{fig:cpu_compare} (a) the CPU time (in seconds) with respect to the total number of time steps $N$ for both schemes with $M=500$. We observe that for the same $N$, our scheme uses much less CPU time than the L1 scheme, and the advantage becomes more pronounced as $N$ increases. The CPU time of the proposed scheme grows approximately linearly with respect to $N$, whereas the computational cost of the original L1 scheme exhibits nearly quadratic growth. This observation is fully consistent with the theoretical complexity estimates given in Sect.~\ref{section3}, namely $\mathcal{O}(MN N_{\mathrm{exp}})$ for the proposed scheme and $\mathcal{O}(MN^2)$ for the L1 scheme.  In addition, Figure~\ref{fig:cpu_compare} (b) demonstrates that the number of exponentials satisfies $N_{\mathrm{exp}} = \mathcal{O}(\log N)$, which further validates the theoretical complexity analysis of the proposed method in Sect.~\ref{section3}. Consequently, the proposed schemes can significantly reduce the computational costs while preserving the same global convergence rate as the L1 scheme, thereby improving the computational efficiency.
    \item For purposes of comparison, Table~\ref{table:errors2.2} lists the $H^1$ errors of both the proposed scheme (\ref{scheme}) and the scheme from \cite{cao2020}. The results show that our scheme achieves second-order convergence in the $H^1$ norm, whereas the method in \cite{cao2020} is only first-order convergent in the same norm.
    \end{enumerate}
\begin{figure}[ht]
	\begin{minipage}[t]{0.4\linewidth}
  \centering
  \includegraphics[width=\textwidth]{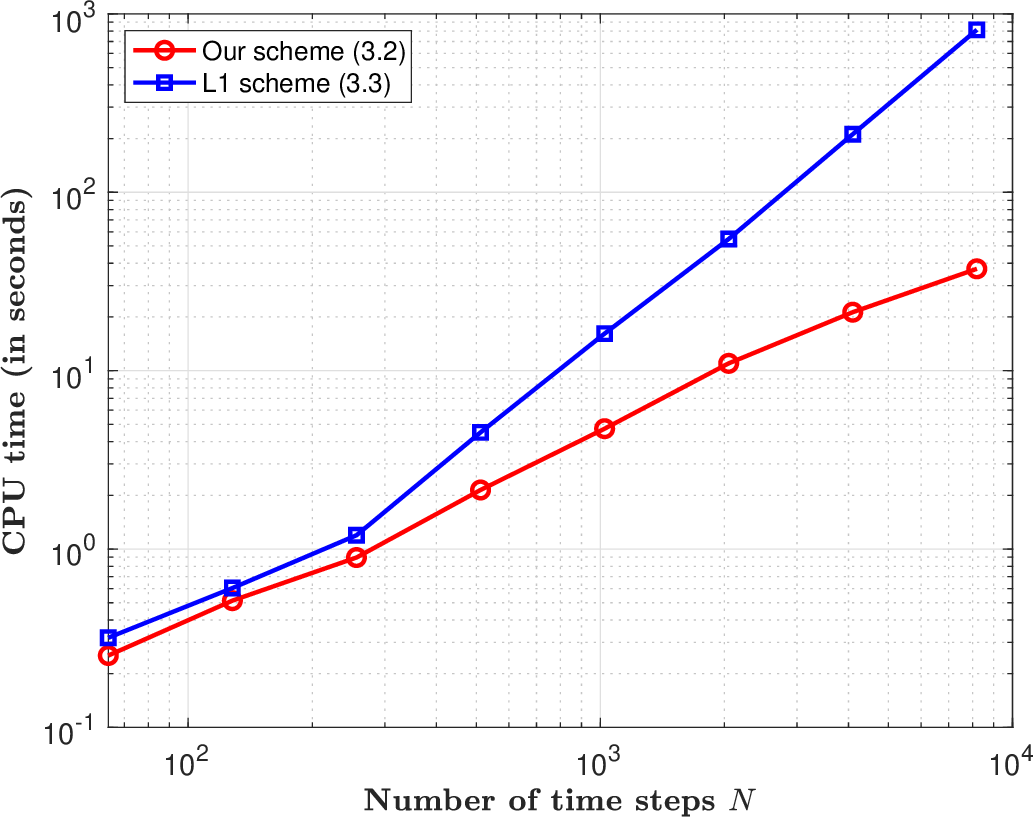}
  \centerline{(a)}
	\end{minipage}%
    \hspace{0.05\linewidth} 
	\begin{minipage}[t]{0.38\linewidth}
   \centering
   \includegraphics[width=\textwidth]{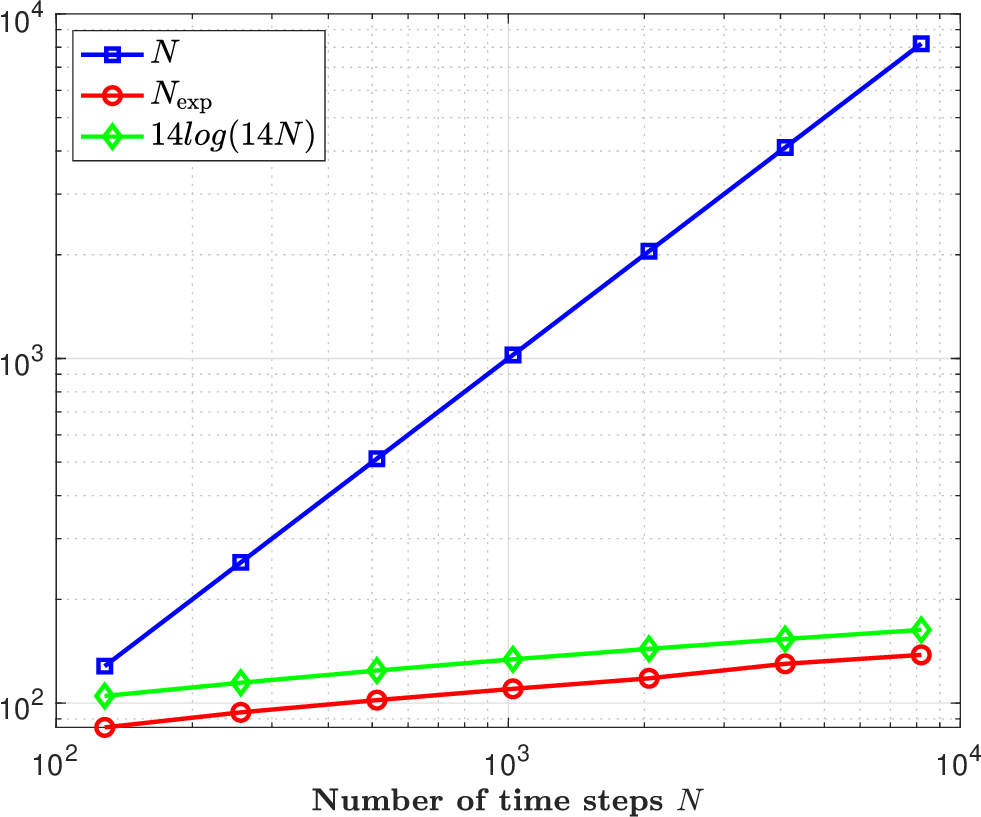}
   \centerline{(b)}
	\end{minipage}
\caption{(Left) A comparison of CPU time in seconds as a function of time steps $N$. (Right) SOE number $N_{exp}$ versus the time steps $N$.}\label{fig:cpu_compare}
\end{figure}
\begin{table}[htbp]
\centering
\caption{A comparison of $H^1$ error and convergence order in time with $\lambda=1$, $\delta=1.8$ and $M=2000$.}
\label{table:errors2.2}
\begin{tabular}{c c c c c c c c c}
\toprule
& \multicolumn{4}{c}{$\alpha=0.25$} 
& \multicolumn{4}{c}{$\alpha=0.5$} \\
\cmidrule(lr){2-5} \cmidrule(lr){6-9}
 & \multicolumn{2}{c}{Our scheme (\ref{scheme})} & \multicolumn{2}{c}{The scheme in \cite{cao2020}}   
 & \multicolumn{2}{c}{Our scheme (\ref{scheme})} & \multicolumn{2}{c}{The scheme in \cite{cao2020}} \\
\cmidrule(lr){2-3} \cmidrule(lr){4-5} \cmidrule(lr){6-7} \cmidrule(lr){8-9}
\multirow{1}{*}{$N$}
 & Error & Order & Error & Order & Error & Order & Error  & Order \\
\midrule
16 
& 2.6614e-04  & --   
& 1.2196e-05 & -- 
& 2.8276e-04 & -- 
& 8.8700e-06 & -- \\
32 
& 6.8095e-05 & 2.0609 
& 6.2261e-06 & 0.9700 
& 7.5985e-05 & 1.9868 
& 4.4926e-06 & 0.9814 \\
64  
& 1.6675e-05  & 2.0769 
& 3.1398e-06 & 0.9877 
& 1.9763e-05 & 1.9881 
& 2.2559e-06 & 0.9938 \\
\bottomrule
\end{tabular}
\end{table}

\noindent\textbf{\underline{Case 2}}

 We consider \eqref{Equation} with $\Omega=(0,1)$, $T=2$, $\phi(x)=\sin (\pi x^2)$ and 
\begin{align}
f(x,t) = & ( -\lambda (t^\delta + 1) + \delta t^{\delta - 1} + \frac{\Gamma(\delta+1)}{\Gamma(\delta-\alpha+1)} t^{\delta-\alpha} ) e^{-\lambda t} \sin(\pi x^2) \notag\\
& - ( -2\pi x\cos(\pi x^2)+2\pi\cos(\pi x^2)-4\pi^2x^2\sin(\pi x^2) ) (t^\delta + 1) e^{-\lambda t}.\notag
\end{align}
The exact solution is then given by $u(x,t) = e^{-\lambda t} (t^\delta + 1) \sin(\pi x^2)$ (see e.g. \cite{cao2020,dwivedi2024,stynes2017,zhu2019}). In what follows, we present our numerical results. 

Again, the maximum $L^2$ errors, together with the corresponding temporal and spatial convergence orders, are reported in Tables~\ref{table:errors2t} and~\ref{table:errors2s}. The numerical results clearly demonstrate that both the proposed scheme~\eqref{scheme} and the L1 scheme~\eqref{L1scheme} achieve second-order convergence in both time and space for all tested cases, which is in perfect agreement with Theorem~\ref{converge}.

Figure~\ref{fig:cpu_compare2} presents the CPU time (in seconds) of both schemes as a function of the number of time steps $N$. It can be observed that the proposed scheme exhibits an almost linear computational complexity with respect to $N$ and is significantly faster than the direct L1 scheme. Moreover, compared with the direct L1 difference scheme, the proposed method attains the same level of accuracy for a prescribed tolerance while substantially reducing the computational cost in terms of elapsed CPU time, particularly when the number of time steps $N$ becomes large. This behavior is fully consistent with the observations in Case~1 and further confirms the efficiency advantage of the proposed scheme.

\begin{table}[htbp]
\centering
\setlength{\tabcolsep}{5pt}   
\caption{A comparison of maximum $L^2$ error and convergence order in time with $\lambda=1$, $\delta=1.8$ and $M=2000$.}
\label{table:errors2t}
\begin{tabular}{c c c c c c c c c}
\toprule
& \multicolumn{4}{c}{$\alpha=0.25$} 
& \multicolumn{4}{c}{$\alpha=0.5$} \\
\cmidrule(lr){2-5} \cmidrule(lr){6-9}
 & \multicolumn{2}{c}{Our scheme \eqref{scheme}} & \multicolumn{2}{c}{L1 scheme \eqref{L1scheme}}   & \multicolumn{2}{c}{Our scheme \eqref{scheme}} & \multicolumn{2}{c}{L1 scheme \eqref{L1scheme}} \\
\cmidrule(lr){2-3} \cmidrule(lr){4-5} \cmidrule(lr){6-7} \cmidrule(lr){8-9}
\multirow{1}{*}{$N$} & Error & Order & Error & Order & Error & Order & Error & Order \\
\midrule
16
& 8.9883e-02 &  --  & 9.0059e-02 & --   & 9.1053e-03 & --   & 9.1303e-03 & -- \\
32
& 2.1358e-03 & 2.1727  & 2.1458e-03 & 2.1686   & 2.1969e-03 & 2.1496   & 2.2128e-03 & 2.1428 \\
64
& 5.3281e-04 & 2.0496   & 5.3792e-04 & 2.0424   & 5.6194e-04 & 2.0127   & 5.7149e-04 & 1.9985 \\
\bottomrule
\end{tabular}
\end{table}

\begin{table}[htbp]
\centering
\setlength{\tabcolsep}{5pt}   
\caption{A comparison of maximum $L^2$ error and convergence order in space with $\lambda=1$, $\delta=1.8$ and $N=1000$.}
\label{table:errors2s}
\begin{tabular}{c c c c c c c c c}
\toprule
& \multicolumn{4}{c}{$\alpha=0.25$} 
& \multicolumn{4}{c}{$\alpha=0.5$} \\
\cmidrule(lr){2-5} \cmidrule(lr){6-9}
 & \multicolumn{2}{c}{Our scheme \eqref{scheme}} & \multicolumn{2}{c}{L1 scheme \eqref{L1scheme}}   & \multicolumn{2}{c}{Our scheme \eqref{scheme}} & \multicolumn{2}{c}{L1 scheme \eqref{L1scheme}} \\
\cmidrule(lr){2-3} \cmidrule(lr){4-5} \cmidrule(lr){6-7} \cmidrule(lr){8-9}
\multirow{1}{*}{$M$} & Error & Order & Error & Order & Error & Order & Error & Order \\
\midrule
20
& 2.2664e-03 & -- 
& 2.2664e-03 & -- 
& 2.2122e-03 & -- 
& 2.2122e-03 & -- \\
40
& 5.6609e-04 & 2.0010 
& 5.6609e-04 & 2.0010 
& 5.5217e-04 & 2.0012 
& 5.5215e-04 & 2.0012 \\
80
& 1.4096e-04 & 2.0053 
& 1.4095e-04 & 2.0053 
& 1.3710e-04 & 2.0093 
& 1.3708e-04 & 2.0093 \\
\bottomrule
\end{tabular}
\end{table}

\begin{figure}[ht]
  \centering
  \includegraphics[width=0.5\textwidth]{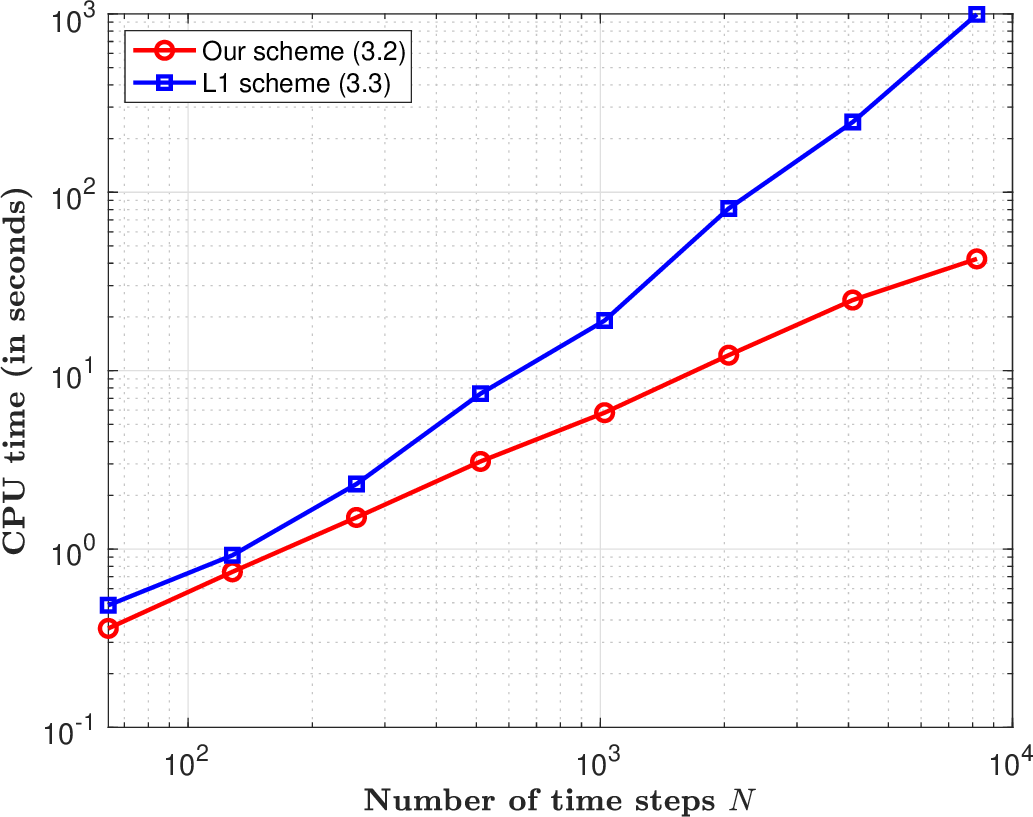}
\caption{ A comparison of CPU time in seconds as a function of time steps $N$.}\label{fig:cpu_compare2}
\end{figure}

\noindent\textbf{\underline{Case 3}}

We consider \eqref{Equation} with $\Omega=(0,1)$, $T=2$, $\phi(x)=x^4(1-x)^4$ and 
\begin{align}
f(x,t) = & ( -\lambda (e^{-x}t^\delta + 1) + \delta e^{-x} t^{\delta - 1} + \frac{\Gamma(\delta+1)}{\Gamma(\delta-\alpha+1)}e^{-x} t^{\delta-\alpha} ) e^{-\lambda t}x^4(1-x)^4 \notag\\
& - ( 12x^2(1-x)^2-4x^3(1-x)^3(3-2x)) (e^{-x}t^\delta + 1) e^{-\lambda t}\notag\\
& - (2x^4(1-x)^4-8x^3(1-x)^3(1-2x))e^{-x}t^\delta e^{-\lambda t}.\notag
\end{align}
The exact solution is then given by $u(x,t) = e^{-\lambda t} (e^{-x}t^\delta + 1)x^4(1-x)^4$ (see e.g.\cite{gao2012,jiang2017}). In what follows, we present our numerical results.    
\begin{table}[htbp]
\centering
\setlength{\tabcolsep}{5pt}   
\caption{A comparison of maximum $L^2$ error and convergence order in time with $\lambda=1$, $\delta=1.8$ and $M=2000$.}
\label{table:errors3t}
\begin{tabular}{c c c c c c c c c}
\toprule
& \multicolumn{4}{c}{$\alpha=0.25$} 
& \multicolumn{4}{c}{$\alpha=0.5$} \\
\cmidrule(lr){2-5} \cmidrule(lr){6-9}
 & \multicolumn{2}{c}{Our scheme \eqref{scheme}} & \multicolumn{2}{c}{L1 scheme \eqref{L1scheme}}   & \multicolumn{2}{c}{Our scheme \eqref{scheme}} & \multicolumn{2}{c}{L1 scheme \eqref{L1scheme}} \\
\cmidrule(lr){2-3} \cmidrule(lr){4-5} \cmidrule(lr){6-7} \cmidrule(lr){8-9}
\multirow{1}{*}{$N$} & Error & Order & Error & Order & Error & Order & Error & Order \\
\midrule
16
& 2.732e-05 & --     & 2.738e-05 & --     & 2.772e-05 & --     & 2.781e-05 & -- \\
32
& 6.843e-06 & 2.0932     & 6.857e-06 & 2.0934     & 7.023e-06 & 2.0760     & 7.050e-06 & 2.0751 \\
64
& 1.696e-06 & 2.0593     & 1.708e-06 & 2.0516     & 1.795e-06 & 2.0141     & 1.821e-06 & 1.9982 \\
\bottomrule
\end{tabular}
\end{table}
\begin{table}[htbp]
\centering
\setlength{\tabcolsep}{5pt}   
\caption{A comparison of maximum $L^2$ error and convergence order in time with $\lambda=1$, $\delta=1.8$ and $N=1000$.}
\label{table:errors3s}
\begin{tabular}{c c c c c c c c c}
\toprule
& \multicolumn{4}{c}{$\alpha=0.25$} 
& \multicolumn{4}{c}{$\alpha=0.5$} \\
\cmidrule(lr){2-5} \cmidrule(lr){6-9}
 & \multicolumn{2}{c}{Our scheme \eqref{scheme}} & \multicolumn{2}{c}{L1 scheme \eqref{L1scheme}}   & \multicolumn{2}{c}{Our scheme \eqref{scheme}} & \multicolumn{2}{c}{L1 scheme \eqref{L1scheme}} \\
\cmidrule(lr){2-3} \cmidrule(lr){4-5} \cmidrule(lr){6-7} \cmidrule(lr){8-9}
\multirow{1}{*}{$M$} & Error & Order & Error & Order & Error & Order & Error & Order \\
\midrule
20
& 1.3541e-01 & -- 
& 1.3541e-01 & -- 
& 1.3331e-01 & -- 
& 1.3331e-01 & -- \\
40
& 3.3352e-02 & 2.0206 
& 3.3352e-02 & 2.0206 
& 3.2839e-02 & 2.0210 
& 3.2839e-02 & 2.0210 \\
80
& 8.3053e-03 & 2.0054 
& 8.3053e-03 & 2.0054 
& 8.1767e-03 & 2.0056 
& 8.1765e-03 & 2.0056 \\
\bottomrule
\end{tabular}
\end{table}

Tables \ref{table:errors3t} and \ref{table:errors3s} present the numerical results for our scheme \eqref{scheme} and L1 scheme \eqref{L1scheme}, which show that our scheme has the same convergence order as the direct L1 scheme, in particular the second-order accuracy in both time and space, which is in perfect agreement with Theorem \ref{converge}. Furthermore, Figure~\ref{fig:cpu_compare3} presents a comparison of the CPU time (in seconds) for the different schemes. Similar to the observations in Cases~1 and~2, essentially the same conclusions are obtained, which further demonstrate both the accuracy and efficiency of the proposed method. Again, the above numerical tests demonstrate our theoretical analysis.
\begin{figure}[ht]
  \centering
  \includegraphics[width=0.5\textwidth]{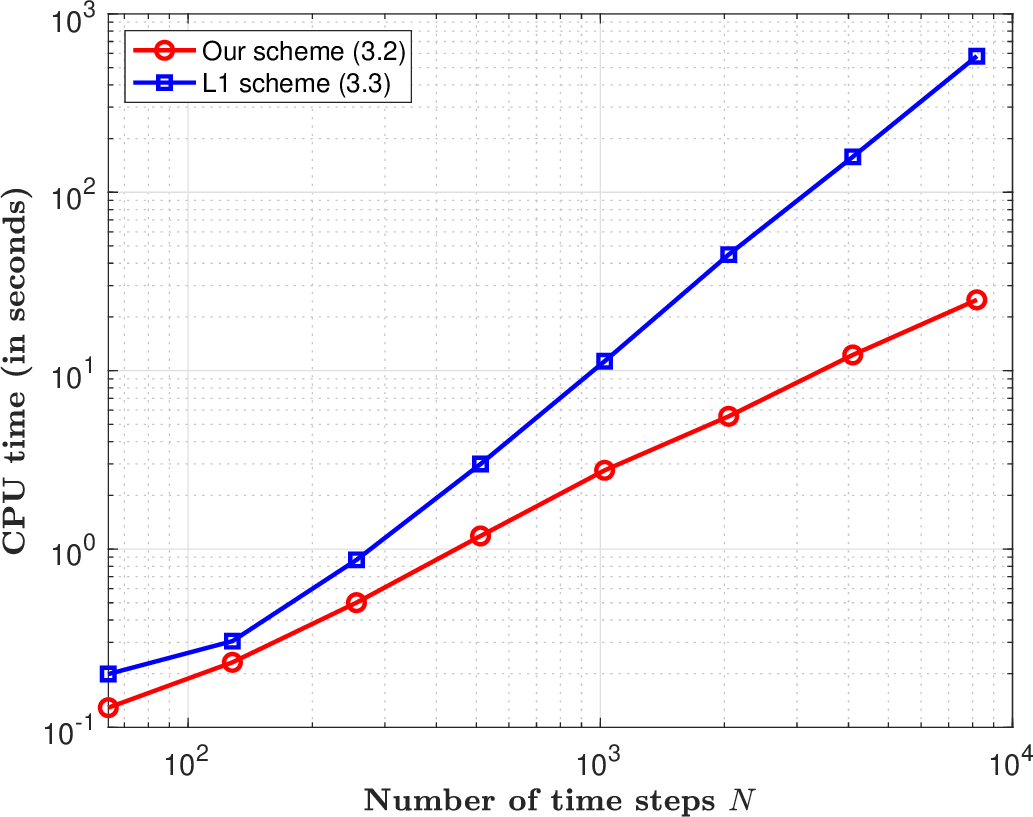}
\caption{ A comparison of CPU time in seconds as a function of time steps $N$.}\label{fig:cpu_compare3}
\end{figure}

\section{Concluding remarks}\label{section5}
In this paper, we develop a second-order finite difference scheme for the tempered time-fractional derivative exhibiting weak regularity at the initial time. The proposed scheme employs piecewise linear interpolation over each small temporal subinterval and yields a fully discrete approximation for the advection–dispersion equation based on fast sum-of-exponentials approximation. To enhance the overall accuracy, temporal differencing is carried out at half-time levels, while the spatial derivatives are approximated by centered differences. In contrast to earlier works, the proposed scheme achieves second-order convergence in both time and space, reduces the storage requirements and the computational costs while maintaining the
same global convergence order. The numerical results were presented to confirm the theoretical analysis results and demonstrate the developed technique is useful and effective for solving the considered model.
\vspace{18pt}

\noindent {\LARGE \bf Declarations}

\vspace{12pt}

\noindent{\bf Conflict of interest}\;\;
 The authors declare that they have no conflict of interest.


\vspace{18pt}


\begin{thebibliography}{1}

\bibitem{Alikhanov2024_1}
A.A. Alikhanov, M.S. Asl, C. Huang, A. Khibiev, A second-order difference scheme for the nonlinear time-fractional diffusion-wave equation with generalized memory kernel in the presence of time delay, Journal of Computational and Applied Mathematics. \textbf{438}, 115515 (2024)

\bibitem{beylkin2005}
G. Beylkin, L. Monzón, On approximation of functions by exponential sums, Applied and Computational Harmonic Analysis. \textbf{19}(1), 17--48 (2005)

\bibitem{Brunner1985}
H. Brunner, The numerical solution of weakly singular Volterra integral equations by collocation on graded meshes, Mathematics of Computation. \textbf{45}, 417--437 (1985)

\bibitem{cao2020}
J. Cao, A. Xiao, W. Bu, Finite difference/finite element method for tempered time fractional advection–dispersion equation with fast evaluation of Caputo derivative, Journal of Scientific Computing. \textbf{83}, 1--27 (2020)

\bibitem{cao2023}
J. Cao, A. Xiao, W. Bu, A fast Alikhanov algorithm with general nonuniform time steps for a two-dimensional distributed-order time–space fractional advection–dispersion equation, Numerical Methods for Partial Differential Equations. \textbf{39}(4), 2885--2908 (2023)

\bibitem{carr2002}
P. Carr, H. Geman, D. Madan, M. Yor, The fine structure of asset returns: An empirical investigation, The Journal of Business. \textbf{75}, 305--332 (2002)

\bibitem{carr2003}
P. Carr, H. Geman, D. Madan, M. Yor, Stochastic volatility for Lévy processes, Mathematical Finance. \textbf{13}, 345--382 (2003)

\bibitem{cartea2007}
Á. Cartea, D. del Castillo-Negrete, Fluid limit of the continuous-time random walk with general Lévy jump distribution functions, Physical Review E. \textbf{76}, 041105 (2007)

\bibitem{chen2007}
C.-M. Chen, F. Liu, I. Turner, V. Anh, A Fourier method for the fractional diffusion equation describing sub-diffusion, Journal of Computational Physics. \textbf{227}(2), 886--897 (2007)

\bibitem{chen2019}
H. Chen, D. Xu and J. Zhou, A second-order accurate numerical method with graded meshes for an evolution equation with a weakly singular kernel, Journal of Computational and Applied Mathematics. \textbf{356}, 152-163 (2019)

\bibitem{chen2018}
S. Chen, J. Shen and L.-L. Wang, Laguerre functions and their applications to tempered fractional differential equations on infinite intervals, Journal of Scientific Computing. \textbf{74}(3), 1286-1313 (2018)

\bibitem{dwivedi2024}
H. K. Dwivedi, Rajeev, A novel fast tempered algorithm with high-accuracy scheme for 2D tempered fractional reaction–advection–subdiffusion equation, Computers \& Mathematics with Applications. \textbf{148}, 371--397 (2024)

\bibitem{dwivedi2025}
H. K. Dwivedi, Rajeev, A novel fast second-order approach with high-order compact difference scheme and its analysis for the tempered fractional Burgers equation, Mathematics and Computers in Simulation. \textbf{213}, 1--25 (2025)

\bibitem{feng2022}
L. Feng, F. Liu, V. V. Anh, S. Qin, Analytical and numerical investigation on the tempered time-fractional operator with application to the Bloch equation and the two-layered problem, Nonlinear Dynamics. \textbf{109}(3), 2041--2061 (2022)

\bibitem{fenwick2024}
J. Fenwick, F. Liu, L. Feng, New insight into the nano-fluid flow in a channel with tempered fractional operators, Nanotechnology. \textbf{35}(8), 085403 (2024)

\bibitem{gao2012}
G. H. Gao, Z. Z. Sun, Y. N. Zhang, A finite difference scheme for fractional sub-diffusion equations on an unbounded domain using artificial boundary conditions, Journal of Computational Physics. \textbf{231}, 2865–2879 (2012)

\bibitem{gracia2018}
J. L. Gracia, E. O'Riordan, M. Stynes, A fitted scheme for a Caputo initial-boundary value problem, Journal of Scientific Computing. \textbf{76}, 703--727 (2018)

\bibitem{gregory1969}
R. T. Gregory and D. L. Karney, A collection of matrices for testing computational algorithms, Wiley-Interscience, London. (1969)
\bibitem{huang2018}
C. Huang, Z. Zhang and Q. Song, Spectral methods for substantial fractional differential equations, Journal of Scientific Computing.  \textbf{74}(3), 1554-1574 (2018)

\bibitem{jiang2017}
S. Jiang, J. Zhang, Q. Zhang and Z. Zhang, Fast evaluation of the Caputo fractional derivative and its applications to fractional diffusion equations,
Communications in Computational Physics.  \textbf{21}(3) , 650-678 (2017)

\bibitem{koponen1995}
I. Koponen, Analytic approach to the problem of convergence of truncated Lévy flights towards the Gaussian stochastic process, Physical Review E. \textbf{52}, 1197--1199 (1995)

\bibitem{krzyzanowski2024}
G. Krzyzanowski, M. Magdziarz, A tempered subdiffusive Black–Scholes model, Fractional Calculus and Applied Analysis. \textbf{27}, 1--29 (2024)

\bibitem{Liu2018}
Z. Liu, X. Li, A Crank–Nicolson difference scheme for the time variable fractional mobile–immobile advection–dispersion equation, Journal of Applied Mathematics and Computing. \textbf{56}, 391--410 (2018)

\bibitem{Liao2018}
H.-L. Liao, D. Li, J. Zhang, Sharp error estimate of the nonuniform L1 formula for linear reaction–subdiffusion equations, SIAM Journal on Numerical Analysis. \textbf{56}(2), 1112--1133 (2018)

\bibitem{Lyu2019}
P. Lyu, S. Vong, A high-order method with a temporal nonuniform mesh for a time-fractional Benjamin–Bona–Mahony equation, Journal of Scientific Computing. \textbf{80}(3), 1607--1628 (2019)

\bibitem{mantegna1994}
R. N. Mantegna, H. E. Stanley, Stochastic process with ultraslow convergence to a Gaussian: The truncated Lévy flight, Physical Review Letters. \textbf{73}, 2946--2949 (1994)

\bibitem{meerschaert2008}
M. M. Meerschaert, Y. Zhang, B. Baeumer, Tempered anomalous diffusion in heterogeneous systems, Geophysical Research Letters. \textbf{35}, L17403 (2008)

\bibitem{morgado2019}
L. Morgado, L. Morgado, Modelling time-of-flight transient currents with time-fractional diffusion equations, Progress in Fractional Differentiation and Applications. \textbf{5}, 1--10 (2019)

\bibitem{nong2024}
L. Nong, Q. Yi and A. Chen, A fast second-order ADI finite difference scheme for the two-dimensional time-fractional cattaneo equation with spatially variable coefficients, Fractal and Fractional. \textbf{8}(8), 453 (2024)


\bibitem{qiao2022}
L. Qiao, J. Guo, W. Qiu, Fast BDF2 ADI methods for the multi-dimensional tempered fractional integrodifferential equation of parabolic type, Computers \& Mathematics with Applications. \textbf{123}, 89--104 (2022)

\bibitem{RaviKanth2018}
A. S. V. Ravi Kanth and S. Deepika, Application and analysis of spline approximation for time fractional mobile–immobile advection-dispersion equation, Numerical Methods for Partial Differential Equations.  \textbf{34}(5), 1799-1819 (2018)

\bibitem{Vabishchevich2022}
P. N. Vabishchevich, Numerical solution of the Cauchy problem for Volterra integro-differential equations with difference kernels, Applied Numerical Mathematics. \textbf{174}, 177--190 (2022)

\bibitem{Vabishchevich2023}
P. N. Vabishchevich, Numerical-analytical methods for solving the Cauchy problem for evolutionary equations with memory, Lobachevskii Journal of Mathematics. \textbf{44}(10), 4195--4204 (2023)

\bibitem{wang2023}
L. Wang and H. Liang, Superconvergence and postprocessing of collocation methods for fractional differential equations
, Journal of Scientific Computing.  \textbf{97}(2), 29 (2023)

\bibitem{safari2022}
Z. Safari, G. B. Loghmani, M. Ahmadinia, Convergence analysis of an LDG method for time–space tempered fractional diffusion equations with weakly singular solutions, Journal of Scientific Computing. \textbf{91}(2), 68 (2022)

\bibitem{stynes2017}
M. Stynes, E. O’Riordan, J. L. Gracia, Error analysis of a finite difference method on graded meshes for a time-fractional diffusion equation, SIAM Journal on Numerical Analysis. \textbf{55}, 1057--1079 (2017)

\bibitem{sun2020}
L. Sun, H. Qiu, C. Wu, J. Niu and B. X. Hu, A review of applications of fractional advection–dispersion equations for anomalous solute transport in surface and subsurface water, WIREs Water. \textbf{7}(4), e1448 (2020)

\bibitem{sun2006}
Z.-Z. Sun, A fully discrete scheme for a diffusion-wave system, Applied Numerical Mathematics. \textbf{56}, 193--209 (2006)

\bibitem{wang2022}
H. Wang, C. Li, Fast difference scheme for a tempered fractional Burgers equation in porous media, Applied Mathematics Letters. \textbf{132}, 108143 (2022)

\bibitem{xia2013}
Y. Xia, J. Wu, Y. Zhang, Tempered time-fractional advection–dispersion equation for modeling non-Fickian transport, Advances in Water Science. \textbf{24}(3), 349--357 (2013)

\bibitem{zhang2021}
H. Zhang, X. Jiang and F. Liu, Error analysis of nonlinear time fractional mobile/immobile advection-diffusion equation with weakly singular solutions, Fractional Calculus and Applied Analysis. \textbf{24}(1), 202-224 (2021)


\bibitem{zhang2011}
Y. Zhang, M. M. Meerschaert, Gaussian setting time for solute transport in fluvial systems, Water Resources Research. \textbf{47}, W08601 (2011)

\bibitem{zhang2012}
Y. Zhang, M. M. Meerschaert, A. Packman, Linking fluvial bed sediment transport across scales, Geophysical Research Letters. \textbf{39}, L20404 (2012)

\bibitem{zhang2014}
Y. Zhang, L. Chen, D. Reeves, H. Sun, A fractional-order tempered-stable continuity model to capture surface water runoff, Journal of Vibration and Control. \textbf{22}, 1993--2003 (2016)

\bibitem{zhao2020}
L. Zhao, C. Li, F. Zhao, Efficient difference schemes for the Caputo-tempered fractional diffusion equations based on polynomial interpolation, Communications on Applied Mathematics and Computation. \textbf{2}, 534--558 (2020)

\bibitem{zhao2023}
T. Zhao, Efficient spectral collocation method for tempered fractional differential equations, Fractal and Fractional. \textbf{7}(3), 277 (2023)

\bibitem{zhao2024}
T. Zhao, L. Zhao, Efficient Jacobian spectral collocation method for spatio-dependent temporal tempered fractional Feynman–Kac equation, Communications on Applied Mathematics and Computation. \textbf{6}, 1--25 (2024)

\bibitem{zhao2021}
Y.-L. Zhao, X.-M. Gu, M. Li, H.-Y. Jian, Preconditioners for all-at-once system from the fractional mobile/immobile advection–diffusion model, Journal of Applied Mathematics and Computing. \textbf{65}, 669--691 (2021)

\bibitem{zheng2025}
Z.-Y. Zheng, Y.-M. Wang, A fast temporal second-order compact finite difference method for two-dimensional parabolic integro-differential equations with weakly singular kernel, Journal of Computational Science \textbf{87}, 102558 (2025)

\bibitem{zhou2024}
J. Zhou, X.-M. Gu, Y.-L. Zhao, H. Li, A fast compact difference scheme with unequal time-steps for the tempered time-fractional Black–Scholes model, International Journal of Computer Mathematics \textbf{101}(9--10), 989--1011 (2024)

\bibitem{zhu2019}
H. Zhu and C. Xu, A fast high order method for the time-fractional diffusion equation, SIAM Journal on Numerical Analysis \textbf{57}(6), 2829-2849 (2019)

\end{thebibliography}
\end{document}